\documentclass[12pt,leqno,a4paper,article]{amsart}
\usepackage{amsmath, amsthm, verbatim, amsfonts, amssymb, color}
\usepackage{mathrsfs}
\usepackage{dsfont}
\usepackage[a4paper,margin=3cm]{geometry}

\usepackage{graphicx}
 \usepackage{subfigure} 
 
\usepackage{booktabs}
\usepackage{array}
\usepackage{multirow,multicol}
\usepackage{diagbox}
\usepackage[figuresright]{rotating}

\usepackage{CJKutf8}

\usepackage{algorithm,algorithmic}
\usepackage{multirow}
\usepackage{float}

\usepackage{bm}
\usepackage{bbm}

\usepackage{mathtools}
\usepackage{mathrsfs}
\usepackage{bm}
\usepackage{marginnote}
\usepackage{extarrows}
\usepackage{xcolor}

\usepackage{enumitem}
\usepackage{hyperref}
\hypersetup{hypertex=true,
	colorlinks=true,
	linkcolor=blue,
	anchorcolor=blue,
	citecolor=green,
	pdftitle={Overleaf Example},
	pdfpagemode=FullScreen}
\makeatletter
\newcommand{\runum}[1]{\mathrm{\romannumeral #1}}
\newcommand{\Rmnum}[1]{\mathrm{\expandafter\@slowromancap\romannumeral #1@}}
\makeatother

\newcommand\dd{\,\mathrm{d}}

\newcommand{\abs}[1]{\left|#1\right|}

\newcommand{\lipnorm}[1]{\left\Vert #1 \right\Vert_{\operatorname{Lip}}}

\newcommand{\hsprod}[1]{\left\langle #1 \right\rangle_{\mathrm{HS}}}
\newcommand{\hsnorm}[1]{\left\Vert #1 \right\Vert_{\operatorname{HS}}}
\newcommand{\inprod}[1]{\left\langle #1 \right\rangle}

\newcommand{\bbR}{\mathbb{R}}
\newcommand{\bbE}{\mathbb{E}}

\newcommand{\rme}{\mathrm{e}}

\newcommand{\pheq}{\mathrel{\phantom{=}}}

\theoremstyle{plain}
\newtheorem{theorem}{Theorem}[section]

\newtheorem{lemma}[theorem]{Lemma}

\theoremstyle{definition}

\numberwithin{equation}{section}
\renewcommand\labelenumi{\textup{\alph{enumi})}}
\renewcommand\theenumi\labelenumi

\usepackage{url}

\usepackage{marginnote}

\title[Phase transition in EM scheme of an SDE driven by stable noises]{Phase transition in 
the EM scheme of an SDE driven by $\alpha$-stable noises with $\alpha \in (0,2]$}

\begin{document}
\begin{CJK}{UTF8}{gbsn}
	
	\allowdisplaybreaks[4]
	
	\date{\today}
	
	\author[Y. Wang]{Yu Wang}
	\address{Yu Wang: 1. Department of Mathematics,
		Faculty of Science and Technology,
		University of Macau,
		Av. Padre Tom\'{a}s Pereira, Taipa
		Macau, China; \ \ 2. UM Zhuhai Research Institute, Zhuhai, China.}
	\email{yc17447@um.edu.mo}
	
	\author[Y. Xiao]{Yimin Xiao}
	\address{Yimin Xiao: Department of Statistics and Probability, 
		A-413 Wells Hall, 
		Michigan State University, 
		East Lansing, MI 48824, 
		USA.}
	\email{xiao@stt.msu.edu}
	
	\author[L. Xu]{Lihu Xu}
	\address{Lihu Xu: 1. Department of Mathematics,
		Faculty of Science and Technology,
		University of Macau,
		Av. Padre Tom\'{a}s Pereira, Taipa
		Macau, China; \ \ 2. UM Zhuhai Research Institute, Zhuhai, China.}
	\email{lihuxu@um.edu.mo}
	

	\begin{abstract}
		We study in this paper the EM scheme for a family of well-posed critical SDEs with the drift $-x\log(1+\abs{x})$ and $\alpha$-stable noises. 
		Specifically, we find that when the SDE is driven by a rotationally symmetric $\alpha$-stable processes with $\alpha=2$ (i.e. Brownian motion), the EM scheme is bounded in the $L^2$ sense uniformly w.r.t. the time. 
		In contrast, if the SDE is driven by a rotationally symmetric $\alpha$-stable process with $\alpha \in (0,2)$, all the $\beta$-th moments, 
		with $\beta \in (0,\alpha)$, of the EM scheme blow up. This demonstrates a phase transition phenomenon as $\alpha \uparrow 2$. We verify our results by simulations. 
	\end{abstract}
	
	\maketitle
	
	\tableofcontents
	
	\noindent


\section{Introduction}
The following stochastic differential equation (SDE) on $\mathbb{R}^d$ has been extensively studied for several decades:
	\begin{equation}
		\label{eq:initial SDE}
		\dd X_t = f(X_t) \dd t + g(X_t) \dd L_t\ , \quad X_0 = x_0,
	\end{equation}
where $x_0 \in \bbR^d$, $f:\bbR^d \to \bbR^d , g: \bbR^d \to \bbR^{d\times d}$ satisfy certain regularity conditions, and the process 
$(L_t, t\geqslant 0)$ is a $d$-dimensional, rotationally invariant $\alpha$-stable L\'evy process with $\alpha \in (0,2]$. We refer to the books
\cite{MR2512800,MR1083357,MR2160585} for systematic accounts on SDEs driven by L\'evy processes and to \cite{arapostathis2019uniform, 
MR3390235, bao2022coupling,baoyuan2017,bao2011comparison, bao2012stochastic, chen2016heat,dang2024eulermaruyama,
dong2020irreducibility,MR2030412,KUHN20192654, WOS:000343456100011, deng2023wasserstein,wang2012exponential} for more 
recent developments.  

The Euler-Maruyama (EM) scheme of  SDE (\ref{eq:initial SDE}) has also been studied by many authors,
see \cite{kloeden1992stochastic,KUHN20192654,higham2002strong, yuan2004convergence,WOS:001028737600001,yuan2008note, chen2023approximation,WOS:000465541700003,MR2972586,bao2019convergence,MR3660882}. 
In particular, when $f$ is Lipschitz and $g$ is bounded Lipschitz, by a standard method for proving the existence and uniqueness 
of SDE's strong solution, it can be shown that the EM scheme in a finite time interval $[0,T]$ strongly converges to \eqref{eq:initial SDE}, 
see for instance \cite{MR1121940,MR2001996,MR2512800,MR2160585,MR1083357}.
	
{It is well known that, even though a (stochastic) dynamics system 
is well posed, its numerical schemes may blow up, }see \cite{kloeden1992stochastic,KUHN20192654, MR2795791,higham2003exponential,bao2012stochastic}. 
In particular, for SDE \eqref{eq:initial SDE} with a 
Brownian motion noise, i.e., $\alpha=2$, if the drift coefficient $f$ is a polynomial like $-x |x|^\theta$ with $\theta>0$, 
it has a unique strong solution with finite exponential moment. In contrast, its EM scheme has a blowing-up $L^p(\mathbb P)$ 
norm with $p \geqslant 1$ as the step size $\eta \rightarrow 0$, see for instance \cite{MR2795791, 
higham2003exponential,hutzenthaler2013divergence}. When the driven noise is an $\alpha$-stable process 
with $\alpha \in (0,2)$,  due to the heavy tailed property of stable distribution, one may expect that the corresponding 
EM scheme will blow up. One of the main motivation of the present paper is to verify this conjecture. 

In this paper, we shall mainly consider SDE (\ref{eq:initial SDE}) with a special non-Lipschitz drift $-x\log(1+|x|)$, 
which lies at the boundary between the cases $-x |x|^\theta$ and $-x$. To the best of our knowledge, the behavior of the 
EM scheme has not been studied. Our main results, Theorems \ref{uniform bound of EM_BM}--\ref{thm 3}, show that the EM scheme is uniformly bounded 
for $\alpha=2$, but blows up for $\alpha \in (0,2)$. This demonstrates a phase transition as $\alpha \uparrow 2$.  

\subsection{The drift $-x \log(1+|x|)$ and $\alpha \in (0,2]$: phase transition}
In order to demonstrate our idea, we assume for simplicity that $g(x)=I_d$ with $I_d$ being $d \times d$ identity matrix. 
We expect that our results can be extended to the case in which $g$ is bounded Lipschitz and nondegenerate.

Let us consider the following SDE on $\mathbb{R}^d$: 
	\begin{equation}
		\label{eq:critical sde}
		\dd X_t = -X_t \log( 1 + \abs{ X_t } ) \dd t + \dd L_t\ , \quad X_0 = x_0, 
	\end{equation}
where $L_t$ is a $d$-dimensional rotationally symmetric $\alpha$-stable L\'evy process with $\alpha \in (0,2]$. 
Thanks to the dissipation of the drift term, we can show by a standard method that SDE \eqref{eq:critical sde} 
admits a unique strong solution, see Appendix \ref{appendix} below. 

The Euler-Maruyama (EM) scheme of SDE \eqref{eq:critical sde} is give by: for $k \in \mathbb Z_{+}$, 
	\begin{equation}
		\label{eq:critical EM}
		Y_{k+1} = Y_{k} - \eta Y_{k} \log(1+\abs{Y_k}) + (L_{(k+1)\eta} - L_{k\eta} ), \quad Y_0 = x_0,
	\end{equation} 
where $\eta>0$ is the step size. We shall show that for $\alpha \in (0,2)$, the above EM scheme $({ Y_n}, n \ge 1)$ 
blows up in $L^\beta(\mathbb P)$ for any $\beta \in (0, 2)$ as $\eta \rightarrow 0$; while for $\alpha=2$, i.e. 
$L_t$ is a standard $d$-dimensional Brownian motion, 
$({ Y_n}, n \ge 1)$  is uniformly bounded in $L^2(\mathbb P)$. This demonstrates a phase transition as $\alpha \uparrow 2$. 
 \\ \ 
	

We start with the following theorem for the case of $\alpha=2$.   
\begin{theorem}
\label{uniform bound of EM_BM}
	Consider the EM scheme \eqref{eq:critical EM} with $\alpha=2$, i.e. the driven noise is Brownian motion. 
	Then, for any fixed initial value $x_0$, there exist constants $\eta_0 \leqslant \min\{(1+\abs{x_0})^{-2}, \rme^{-5}\}$ and $C>0$ 
	such that for all $\eta \in (0, \eta_0]$, 
	\begin{equation}\label{Eq:Th1.1}
		\sup_{  m \geqslant 0 } \bbE  \abs{ Y_m }^2 \leqslant C.
	\end{equation}
\end{theorem}
	
Next we consider the case $\alpha \in (0,2)$, in which $(L_t,t\geqslant0)$ is a rotationally symmetric $\alpha$-stable process, 
which will be denoted by $(Z_t, t\geqslant0)$. We refer to \cite{ken1999levy} for comprehensive account on L\'evy 
processes, stable and more generally, infinitely divisible distributions. See also \cite{chen2016heat,liu2020gradient,kim2018heat} 
for more recent developments.

It is known (cf. e.g.,  \cite[Theorem 14.3]{ken1999levy} or \cite{WOS:000343456100011,KIM20142479}) 
that the L\'evy measure  $\nu$ of the process $(Z_t, t\geqslant 0)$ is
\[
\nu(\dd z) = \frac{C_{d,\alpha}}{\abs{z}^{d+\alpha}} \dd z,
\]
where constant $C_{d,\alpha}$ is given by
\begin{equation} \label{e:C_d}
C_{d,\alpha} = \alpha 2^{\alpha-1} \pi^{-\frac{d}2} \cdot \frac{\Gamma\big( ({d+\alpha})/{2} \big) }{ \Gamma\big( 1-{\alpha}/{2} \big)}.
\end{equation} 
Even though, for a general $\alpha \in (0, 2)$, the transition density function $p_{\alpha}(t,x)$ of the $\alpha$-stable processes $(Z_t, t\geqslant 0)$ 
does not have an explicit expression, many of its analytic or asymptotic properties have been known.  
In particular, it follows from \cite[Theorem 2.1]{blumenthal1960some}
that $p_{\alpha}(t,x)$ satisfies
 	\begin{equation}
	\label{ineq:heat kernel}
			K(d,\alpha)^{-1} \frac{t}{( t^{1/\alpha} + \abs{x} )^{d+\alpha} } \leqslant p_{\alpha}(t,x) 
			\leqslant K(d,\alpha) \frac{t}{( t^{1/\alpha} + \abs{x} )^{d+\alpha} }, 
			\quad x \in \bbR,
		\end{equation}
where  $K(d,\alpha)\geqslant 1$ is a constant depending on $d$ and $\alpha$.   

By the L\'evy-It\^o decomposition (cf. \cite[Chapter 4]{ken1999levy} or \cite[Chapter 2]{MR2512800}), there exists a Poisson 
random measure $P(\dd t, \dd z)$ such that
\[
\dd Z_t = \int_{\{\abs{z}\geqslant1\}} z P(\dd t, \dd z) + \int_{\{\abs{z}<1\}} z \widetilde{P}(\dd t, \dd z),
\]
where $\widetilde{P}(\dd t, \dd z) = P(\dd t, \dd z) - \dd t\ \nu(\dd z)$ is the compensated Poisson random measure. Due to the lack of  
explicit representation for the probability density of the $\alpha$-stable noise $Z_{(n+1)\eta} - Z_{n\eta}$, the numerical simulation becomes 
complicated and computationally expensive. Hence, one often does not use the scheme \eqref{eq:critical EM} directly in practice, see 
 \cite{nolan2020univariate,chambers1976method} for further discussions. Since the $\alpha$-stable 
distribution and the Pareto distribution with parameter $\alpha$ have the same tail behavior and the stable central limit theorem 
(see, e.g. \cite{chen2019multivariate,hall1981two}), we can replace the stable noise $Z_{(n+1)\eta} - Z_{n\eta}$ with i.i.d. random variables with
the  Pareto distribution. More precisely, for any fixed time $T>0$, the EM approximation for the SDE \eqref{eq:critical sde}, denoted by 
mappings $\widetilde{Y}_{k} : \Omega \to \mathbb{R}^d$, is given by   
		\begin{equation}
			\label{e:EM scheme}
			\widetilde{Y}_{k+1} =\widetilde{Y}_{k} - \eta \widetilde{Y}_{k} \log \left( 1+ \big| \widetilde{Y}_{k} \big| \right)
			+ \frac{1}{\sigma} \eta^{1/\alpha} \widetilde{Z}_{k+1} , \quad \widetilde{Y}_{0}:= x_0 ,
		\end{equation}
		for all $k\in \left\{0,1,\ldots,n-1 \right\}$, $n\in \mathbb{N}$, where $\sigma^{\alpha} = \alpha / (s_{d-1} C_{d,\alpha})$ with 
		$C_{d,\alpha}$ defined by \eqref{e:C_d}, $\{\widetilde{Z}_{k}, k=1,2,\dotsc\}$ is a sequence of i.i.d. Pareto-distributed random variables, 
		and the density function $p(z)$ of the Pareto distribution is given by
		\begin{equation}
			p(z) =\frac{\alpha}{s_{d-1}|z|^{\alpha+d}} \mathbbm{1}_{(1,\infty)} (|z|).
		\end{equation}
		Here, $s_{d-1} = 2\pi^{\frac{d}2} / \Gamma(\frac{d}{2})$ represents the surface area of the unit sphere $\mathbb{S}^{d-1} \subseteq \bbR^d$. 
		

The following theorem describes the limiting behavior of the EM schemes \eqref{eq:critical EM} and (\ref{e:EM scheme}) when $\alpha \in (0,2)$. 
In Theorems \ref{thm 2}  and \ref{thm 3}, 
 $\kappa_\alpha$ and $\delta_\alpha$ are the constants, depending on $\alpha$, given in Lemmas \ref{Lemma 2.1} and \ref{Lemma 2.2} below.

\begin{theorem} \label{thm 2} 
Let $\alpha \in (0,2)$, $\beta \in (0,\alpha)$, $T \in (0,\infty)$ be constants and let  $\eta=\frac{T}{n}$ be the step size. 
			
			$(\runum{1})$ For the EM scheme \eqref{eq:critical EM}, define $K_1 = 2(\rme^{\delta_\alpha/\beta}+2)/T$. We assume that 
			$T$ and $n$ are large enough such that
			\[
			\frac{T}n \leqslant 1,  \quad   K_1 < \frac{\delta_\alpha - \log \delta_\alpha}{\alpha-\beta},  
			\quad
			\rme^{n K_1 } \geqslant \abs{x_0}\big(1+\log(1+\abs{x_0}) \big).
			\]
			Then, there exist a constant $C>0$ and a sequence of non-empty events $\Omega_n \subseteq \Omega$, $n\in \mathbb{N}$ with
			\[
			\mathbb{P}(\Omega_n) \geqslant \frac{C\rme^{-\alpha n K_1}} { n \delta_\alpha^{n}   }
			\quad \text{and} \quad
			\abs{ Y_n(\omega) } \geqslant \exp\left\{ nK_1+(n-1)\frac{\delta_\alpha}{\beta} \right\}
			\]
			for all $\omega\in\Omega_n$ and all $n\in \mathbb{N}$ large enough. Consequently, we have 
			\[
			\lim_{n\to \infty} \bbE\abs{Y_n}^{\beta} = \infty.
			\]
			
			$(\runum{2})$ For the EM scheme \eqref{e:EM scheme}, define $K_2 = 2(\rme^{\kappa_\alpha/\beta}+2/\sigma)/T$. 
			We assume that $T$ and $n$ are large enough such that
			\[
			\frac{T}n \leqslant 1, \quad 	K_2 < \frac{\kappa_\alpha - \log \kappa_\alpha}{\alpha-\beta},  
			\quad
			\rme^{n K_2 } \geqslant \abs{x_0}\big(1+\log(1+\abs{x_0}) \big).
			\]
			Then, there exist a constant $C>0$ and a sequence of non-empty events $\widetilde{\Omega}_n \subseteq \Omega$, $n\in \mathbb{N}$ with 
			\[
			\mathbb{P}(\widetilde{\Omega}_n) \geqslant \frac{C\rme^{-\alpha n K_2}} { n \kappa_\alpha^{n}   }
			\quad \text{and} \quad
			\big|{ \widetilde{Y}_n(\omega) }\big| \geqslant \exp\left\{ nK_2+(n-1)\frac{\kappa_\alpha}{\beta} \right\}
			\]
			for all $\omega\in\widetilde{\Omega}_n$ and all $n\in \mathbb{N}$ large enough.  Consequently, we have 
			\[
			\lim_{n\to \infty} \bbE\big|{\widetilde{Y}_n}\big|^{\beta} = \infty.
			\]
		\end{theorem}
		
\subsection{The polynomial growth drift and $\alpha \in (0,2)$: blow up}	
By the same method for showing Theorem \ref{thm 2}, we can also prove that the EM scheme of SDE \eqref{eq:initial SDE} 
blows up if the drift has a $p$-order polynomial growth with $p>1$.  More precisely, we assume that $f(x)$ and $g(x)$ in 
SDE \eqref{eq:initial SDE} satisfy the following condition: There exist constants $\gamma > \lambda > 1$ and $H\geqslant 1$ such that
	\begin{equation}
		\label{assump_f_g}
		\max\{ \abs{ f(x) } , \abs{ g(x) } \} \geqslant \frac1H \abs{ x }^{\gamma}, 
		\text{ and },
		\min\{ \abs{ f(x) } , \abs{ g(x) } \} \leqslant H \abs{ x }^{\lambda}
		\tag*{{\bf (A)}}
	\end{equation}
for all $\abs{x} \geqslant H$. 

Assumption {\bf (A)} is similar to the condition in  \cite[Theorem 1]{MR2795791}. The EM scheme of the corresponding SDE is 
\begin{equation}
		\label{eq:initial EM}
		{Y}_{k+1} ={Y}_k + \eta f\big({Y}_k\big) +g\big( {Y}_k \big) (Z_{(k+1)\eta} - Z_{k\eta}), \quad {Y}_0=x_0, 
	\end{equation}	
with step size $\eta = \frac{T}{n}$ and $k\in\{0,1,\dotsc,n-1\}$. In practice, it is easier to consider the following EM scheme:
	\begin{equation}
		\label{e:general EM}
		\widetilde{Y}_{k+1} = \widetilde{Y}_k + \eta f\big(\widetilde{Y}_k\big) + \frac1{\sigma} \eta^{1/\alpha} g\big( \widetilde{Y}_k \big) \widetilde{Z}_{k+1}, 
		\quad \widetilde{Y}_0=x_0,
	\end{equation}
where $\{\widetilde Z_k, k=1,2,\dotsc\}$ are i.i.d. Pareto distributed random variables. 

The following theorem shows that for any $T>0$, both EM schemes (\ref{eq:initial EM}) and \eqref {e:general EM} blow up as $\eta \rightarrow 0$. 	
	\begin{theorem}
		\label{thm 3}
		Consider SDE \eqref{eq:initial SDE} under the assumption that $L_t$ is a standard $d$-dimensional  rotationally invariant $\alpha$-stable process 
		with $\alpha\in (0,2)$. We assume that \ref{assump_f_g} holds and  $g(x_0) \neq 0$. Let $T>0$ be an arbitrary number and $\eta=\frac{T}n$. 
		Then, there exists a constant $c >1$ such that 
		
		$(\runum{1})$ for the standard EM scheme \eqref{eq:initial EM}, there exists a sequence of non-empty events $\Omega_n \subseteq \Omega$, 
		$ n\in \mathbb{N}$ such that $ \mathbb{P}(\Omega_n) \geqslant c n^{-c} \delta_\alpha^{-n}$. Furthermore,  $\abs{Y_n(\omega)} 
		\geqslant 2^{\lambda^{n-1}}$ for all $\omega \in \Omega_n$ and all $n\in \mathbb{N}$;
		
		$(\runum{2})$ for the EM scheme \eqref{e:general EM}, there exists a sequence of non-empty events $\widetilde{\Omega}_n \subseteq \Omega$, 
		$ n\in \mathbb{N}$ such that $ \mathbb{P}(\widetilde{\Omega}_n) \geqslant c n^{-c} \kappa_\alpha^{-n}$. Furthermore, $\big|{\widetilde{Y}_n(\omega)}\big| 
		\geqslant 2^{\lambda^{n-1}}$ for all $\omega \in \widetilde{\Omega}_n$ and all $n\in \mathbb{N}$.
		
	Consequently, for any $\beta \in(0,\alpha)$, we have
		\[
		\lim_{n\to \infty} \bbE \abs{Y_n}^{\beta} = \infty, \ \text{ and }, \ \lim_{n\to \infty} \bbE \big|{\widetilde{Y}_n}\big|^{\beta} = \infty.
		\] 
	\end{theorem}

			
	The structure of this paper is outlined as follows: 
	We introduce some useful auxiliary lemmas in the following section.
	In Section \ref{sec 2}, we will prove Theorem \ref{uniform bound of EM_BM}. 
	The proofs of Theorems \ref{thm 2} and \ref{thm 3} are presented in Section \ref{sec 3}. In Section \ref{sec:simulation}, 
	we provide numerical simulations that illustrate the convergence and divergence of EM schemes for $d = 1$.  Finally, in 
	Appendix \ref{appendix}, we establish the existence and uniqueness of the strong solution of SDE \eqref{eq:critical sde}.   
	
\section{ Preliminary and Auxiliary Lemmas}\label{sec:auxiliary}

	To prove Theorem \ref{uniform bound of EM_BM}, we will make use of the following lemmas. Since the proof of Lemma \ref{lem:bound of BM}
	is elementary, it is omitted.
	
	\begin{lemma}
		\label{lem:bound of BM}
		If $N$ follows a standard Gaussian distribution, then for any constants $b\geqslant a \geqslant 0$, there exists a constant $C$ such that
		\begin{equation}
			\mathbb{P}(a \leqslant \abs{N} \leqslant b) \leqslant C(b-a) \rme^{-\frac{a^2}{2}} .
		\end{equation}
	\end{lemma}
	

	As for the Pareto distribution with parameter $\alpha\in(0,2)$, we have the following lemma.
		
	\begin{lemma}\label{Lemma 2.1}
		Suppose that  random vector $\widetilde{Z}: \Omega\to \mathbb{R}^d$ satisfies Pareto distribution, then for all $ z\in (1,\infty)$ we have
		\begin{equation}\label{e:pareto geq x}
			\mathbb{P} (|\widetilde{Z}| \geqslant z) = \frac{1}{z^\alpha}
		\end{equation}
			and there exists a constant $\kappa_{\alpha} \geqslant 1$ such that
					\begin{equation}\label{e:pareto x 2x}
			\mathbb{P} (z\leqslant |\widetilde{Z}| \leqslant 2z) =  \frac{1}{\kappa_{\alpha} z^\alpha},
		\end{equation} 
		where constant $\kappa_\alpha = 2^\alpha/(2^{\alpha}-1)$ which only depends on $\alpha$.
	\end{lemma}
			
	\begin{proof}[Proof of Lemma \ref{Lemma 2.1}]
		For all $z\in (1,\infty)$, and $\alpha \in (0,2)$, we have
		\begin{equation*}
			\mathbb{P} (|\widetilde{Z}| \geqslant z) =\int_{z}^{\infty} \frac{\alpha r^{d-1}}{r^{\alpha+d}} \mathrm{d} r \int_{\mathbb{S}^{d-1}}
			 \frac{1}{s_{d-1}} \mathrm{d} S=\frac{1}{z^\alpha}.
		\end{equation*}
		Similarly, we have 
		\begin{equation*}
			\mathbb{P} (z\leqslant |\widetilde{Z}| \leqslant 2z) =\left(1-\frac{1}{2^\alpha}\right)\frac{1}{z^\alpha} .
		\end{equation*}
		Taking $\kappa_\alpha = 2^{\alpha}/(2^{\alpha}-1)$ gives   \eqref{e:pareto x 2x}.
	\end{proof}	
		
	In addition, thanks to inequality \eqref{ineq:heat kernel}, we also have the following lemma.
			
	\begin{lemma}
	\label{Lemma 2.2}
	Let $Z:\Omega \to \bbR^d$ be a standard $d$-dimensional rotationally invariant $\alpha$-stable distribution 
	with $\alpha \in (0,2)$. Then for all $z \in (1,\infty)$, we have 
		\begin{equation}
			\mathbb{P}(\abs{Z}\geqslant z) \geqslant \frac{s_{d-1}}{2^{d+\alpha} K(d,\alpha)} \cdot \frac1{z^{\alpha}},
		\end{equation}
		and there exists a constant $\delta_\alpha\geqslant 1$ depending on $d$ and $\alpha$  such that for all $z \in (1,\infty)$,
		\begin{equation}
			\mathbb{P}(z \leqslant \abs{Z} \leqslant 2z) \geqslant \frac{1}{\delta_\alpha z^{\alpha}}.
		\end{equation}
		where $\delta_\alpha = \frac {2^{d+\alpha} \alpha K(d,\alpha)}{s_{d-1}} \left( 1-\frac1{2^\alpha} \right)^{-1}$ with 
		$s_{d-1} = 2\pi^{\frac{d}2} / \Gamma(\frac{d}{2})$, and $K(d,\alpha)$ is in \eqref{ineq:heat kernel}. 
	\end{lemma}
	\begin{proof}\ref{Lemma 2.2}
		The proof is analogous to that of Lemma \ref{Lemma 2.1}. For all $z\in(1,\infty)$, by applying inequality \eqref{ineq:heat kernel}, 
	we obtain 
		\[
		\begin{split}
			\mathbb{P}(\abs{Z}\geqslant z) &\geqslant \frac1{K(d,\alpha)} \int_{\abs{x} \geqslant {z}} \!\! \frac{ \dd x }{(1+\abs{x})^{d+\alpha}}  \\
			&\geqslant \frac{s_{d-1}}{2^{d+\alpha} K(d,\alpha)} \int_z^{\infty} \!\! \frac{\dd r}{r^{\alpha + 1}} = \frac{s_{d-1}}{2^{d+\alpha} 
			\alpha K(d,\alpha)} \cdot \frac1{z^{\alpha}},
			\end{split}
		\]
		where the second inequality is due to $\abs{x} \geqslant z > 1$. Similarly,
		\[
			\mathbb{P}(z \leqslant \abs{Z} \leqslant 2z) \geqslant \frac{s_{d-1}}{2^{d+\alpha} \alpha K(d,\alpha)} \left( 1-\frac1{2^\alpha} \right) \cdot \frac1{z^{\alpha}} 
			= \frac{1}{\delta_\alpha z^{\alpha}}.
		\]
		We complete the proof.
	\end{proof}

\section{EM scheme in the case of $\alpha=2$}\label{sec 2}

	In this section, we prove Theorem \ref{uniform bound of EM_BM}.  It is easy to see that the EM scheme of \eqref{eq:critical sde} can be written as
	\begin{equation}
		\label{EM scheme BM}
		Y_{k+1} = Y_{k} - Y_{k} \log( 1 + \abs{Y_{k}} ) \eta + \sqrt{\eta} { N_{k+1}} , \quad Y_0 = x_0, \quad k = 0, 1, 2, \dotsc,
	\end{equation}
	where $\{N_k, k\geqslant1\}$ represents a sequence of i.i.d. standard Gaussian random variables and, for each $k\ge 0$,  $N_{k+1}$ is
	 independent of $(Y_{j}, j \leqslant k)$.
	
\begin{proof}[Proof of Theorem \ref{uniform bound of EM_BM}]
Let $n\in \mathbb{N}$ be arbitrary. Denote
\[
A_0= \Big\{ \omega \in \Omega:\ \sup_{1 \leqslant m \leqslant n} \abs{N_m(\omega)} \in \left[ 0 , \abs{ \log \eta }^2 \right)   \Big\}.
\]
We make the following claim, whose proof will be postponed. 
\begin{equation} \label{e:XlogX1}
\sup_{1 \leqslant m \leqslant n} \abs{Y_m(\omega)} \leqslant \eta^{-1/2}  \abs{\log \eta}^2, \ \ \ \ \omega \in A_0,
\end{equation}
and, for all $p \geqslant 2$, there exists a constant $C_p$ such that
\begin{equation} \label{e:XlogX2}
\mathbb{E} \left[\sup_{1 \leqslant m \leqslant n}|Y_m|^p  \mathbbm{1}_{A^c_0}\right] \leqslant C_p. 
\end{equation} 

Let us prove (\ref{Eq:Th1.1}) first. It follows from  (\ref{e:XlogX1}) and (\ref{e:XlogX2})  that 
\begin{equation}\label{e:EX^2log^2X}
\begin{split}
&\pheq \bbE \left[ \abs{Y_m}^2 \log^2\big( 1+\abs{Y_m} \big) \right] \\
&=\bbE \left[ \abs{Y_m}^2 \log^2\big( 1+\abs{Y_m} \big) \mathbbm{1}_{A_0}\right]+\bbE \left[ \abs{Y_m}^2 \log^2\big( 1+\abs{Y_m} \big) \mathbbm{1}_{A^c_0}\right] \\
&\leqslant \log^2\big(1+\eta^{-1/2} \abs{\log \eta}^2 \big)\bbE \abs{Y_m}^2 + C \\
& \leqslant 4|\log \eta|^2  \bbE \abs{Y_m}^2+C ,
\end{split}
\end{equation} 
where $C$ can be taken as $C_p$ in \eqref{e:XlogX2} with $p = 3$, and the last inequality is by $1+\eta^{-1/2}\abs{\log \eta}^2 \leqslant \eta^{-2}$ for $\eta \leqslant \frac14$. 
On the other hand, 
\begin{equation} \label{e:EXlogX}
\begin{split}
\mathbb E\big[|Y_m|^2 \log(1+|Y_m|)\big] & \geqslant   \mathbb E\big[|Y_m|^2 \log(1+|Y_m|) \mathbbm{1}_{\{|Y_m| \geqslant 2\}}\big] \\
& \geqslant  \mathbb E\big[|Y_m|^2 \mathbbm{1}_{\{|Y_m| \geqslant 2\}}] \\
& \geqslant \mathbb E|Y_m|^2-4.
\end{split} 
\end{equation}
By \eqref{EM scheme BM}, \eqref{e:EX^2log^2X} and \eqref{e:EXlogX},  we see that for all $0 \leqslant m < n$ and $\eta \leqslant \rme^{-5}$, 
\[
			\begin{aligned}
				\bbE \abs{Y_{m+1}}^2 &=\bbE \abs{ Y_m }^2 - 2\eta \bbE \left[\abs{ Y_m }^2 \log(1+|Y_m|)\right] + d \eta 
				+ \eta^2 \bbE \left[ \abs{Y_m}^2 \log^2\big( 1+\abs{Y_m} \big) \right] \\
				&\leqslant \bbE \abs{ Y_m }^2 - 2\eta \bbE \left[\abs{ Y_m }^2\right] +(d+8) \eta + \eta^2 \left[ 4\abs{\log \eta}^2 \bbE \abs{Y_m}^2 + C  \right] \\
				&\leqslant  \big[1 - 2\eta +4\eta^2 \abs{\log \eta}^2 \big] \bbE \abs{ Y_m }^2  + C \eta \\
				&\leqslant \rme^{-\eta} \bbE \abs{ Y_m }^2 + C \eta \\
				&\leqslant \rme^{-(m+1)\eta} \abs{x_0}^2 + C\sum_{k=1}^{m+1} \eta \rme^{(k-(m+1))\eta} \\
				&\leqslant \abs{x_0}^2 + C \rme^{-(m+1)\eta} \int_0^{(m+1)\eta} \rme^x \dd x
 					\leqslant \abs{x_0}^2 + C .
			\end{aligned}
		\] 
	 Since $n\in \mathbb{N}$ is arbitrary,  (\ref{Eq:Th1.1}) in Theorem \ref{uniform bound of EM_BM} clearly holds true. 

It remains to show \eqref{e:XlogX1} and \eqref{e:XlogX2}. 
For proving \eqref{e:XlogX1}, let $m_0 = \min\{ 1\leqslant m \leqslant n: \abs{Y_m} >  \eta^{-1/2}\abs{\log \eta}^2 \}$ and consider
\begin{equation}\label{Eq:Ym}
	Y_{m_0} = Y_{m_0-1}\big( 1-\eta \log(1+\abs{Y_{m_0-1}}) \big) + \sqrt{\eta} N_{m_0}.
\end{equation}
It is easy to check that as $\eta \leqslant \rme^{-5}$
\[
	\eta \log\big( 1+ \eta^{-1/2}\abs{\log \eta}^2 \big) \leqslant 1.
\]
 Since $\abs{x_0} \leqslant \eta^{-\frac12}$, \eqref{e:XlogX1} holds for $m=0$. If $2\abs{Y_{m_0-1}} \leqslant \eta^{-1/2}\abs{\log \eta}^2$, 
 then it follows from (\ref{Eq:Ym}) that for every $\omega \in A_0$,
\[
	\abs{Y_{m_0}} \leqslant \frac{1}2 \eta^{-1/2} \abs{\log \eta}^2  + \sqrt{\eta}  \abs{\log\eta}^2 \leqslant \eta^{-1/2} \abs{\log \eta}^2.
\]
If $2\abs{Y_{m_0-1}} > \eta^{-1/2}\big( \abs{\log \eta}^2\big)$, due the definition of $m_0$, we have
\[
	\begin{aligned}
		\abs{Y_{m_0}} &\leqslant \abs{Y_{m_0-1}} \big( 1-\eta \log(1+{\abs{Y_{m_0-1}}}) \big) + \sqrt{\eta} \abs{N_{m_0}} \\
		&\leqslant \eta^{-1/2} \abs{\log \eta}^2  \Big[ 1-\eta \log\big(1+\eta^{-1/2}\abs{\log \eta}^2/2\big) \Big]  + \sqrt{ \eta}  \abs{\log \eta}^2 \\
		&\leqslant \eta^{-1/2} \abs{\log \eta}^2 .
	\end{aligned}
\]
The last inequality follows from $\eta^{-1/2} \abs{\log\eta}^2 \geqslant 2\rme$ for all $\eta \leqslant \rme^{-2}$. Hence, $m_0$ doesn't exist 
and claim  \eqref{e:XlogX1} holds.

In order to prove \eqref{e:XlogX2}, we split $A^c_0$ into the following disjoint events: 
		\begin{equation*}
			\begin{aligned}
			A_1 &= \Big\{ \omega \in \Omega:\ \sup_{1 \leqslant m \leqslant n} \abs{N_m(\omega)} \in \left[ \abs{ \log \eta }^2 , \eta^{-1} \right)   \Big\} ; \\
			A_k &= \Big\{ \omega \in \Omega:\ \sup_{1 \leqslant m \leqslant n} \abs{N_m(\omega)} \in \left[ ({k-1}){\eta^{-1}} , {k}{\eta^{-1}} \right)  \Big\},  
			\quad 2 \leqslant k \leqslant \lceil \rme^{\frac1{2\eta}} \rceil =: k_0 ; \\
				B_k &= \Big\{ \omega \in \Omega:\ \sup_{1 \leqslant m \leqslant n} \abs{N_m(\omega)} \in \left[ {k}{\eta^{-1}} \rme^{\frac{1}{2\eta}} ,
				 ({k+1}){\eta^{-1}} \rme^{\frac{1}{2\eta}} \right)   \Big\} , \quad 1\leqslant k < \infty.
			\end{aligned}
		\end{equation*}

		Firstly, under $\abs{Y_0} = \abs{x_0} \leqslant {\eta}^{-\frac12}$,  we verify that for every $1 \leqslant k \leqslant \lceil \rme^{\frac1{2\eta}} \rceil$,
		\begin{equation}
			\label{ine:bound on A_k}
			\sup_{1 \leqslant m \leqslant n} \abs{Y_m(\omega)} \leqslant \frac{k}{\eta^{3/2}}\ , \quad \omega \in A_k .
		\end{equation}
		Even though this is similar to the proof of \eqref{e:XlogX1}, we still give the detail here for completeness. 
		Notice that  $ \eta \log\big(1 + k/(2\eta) \big) < \eta \log\big(1 + (\rme^{\frac1{2\eta}}+1)/(2\eta)\big) < 1$. 
		Let $m_0 = \min \{ 1 \leqslant m \leqslant n: \abs{Y_m} > (k+1)/\eta \}$. If $2\abs{ Y_{m_0-1} } \leqslant k\eta^{-\frac32}$, 
		then, due to $\eta \leqslant \rme^{-5} \leqslant  \frac12$,
		\[
			\abs{ Y_{m_0} } \leqslant \abs{ Y_{m_0-1} } + \sqrt{\eta} \abs{N_{m_0}} \leqslant \frac{k}{2\eta^{3/2}} + \frac{k}{\sqrt{\eta}} \leqslant \frac{k}{\eta^{3/2}}.
		\]
		On the other hand, if $2\abs{ Y_{m_0-1} } > k\eta^{-\frac32}$, then
		\[
			\begin{aligned}
				\abs{Y_{m_0}} &\leqslant \abs{Y_{m_0-1}} \big( 1- \eta \log( 1 + k/(2\eta^{3/2}) ) \big) + \sqrt{\eta} \abs{N_{m_0}} \\
				&\leqslant \frac{k}{\eta^{3/2}} ( 1-\eta ) + \frac{k}{\sqrt{\eta}} \leqslant \frac{k}{\eta^{3/2}},
			\end{aligned}
		\]
		where the second inequality comes from the  definition of $m_0$ and $\eta \leqslant 2^{-2/3}$. Hence, $m_0$ doesn't exist 
		and the claim \eqref{ine:bound on A_k} holds. Besides, Lemma \ref{lem:bound of BM} yields that
		\[
		\begin{aligned}
			\mathbb{P}(A_1) &\leqslant C\left(\frac1\eta - \abs{\log \eta}^2 \right)\rme^{-\abs{\log \eta}^2/2} \leqslant {C} \eta^{\frac{\abs{\log \eta}}{2} - 1 } , \\
				\mathbb{P}(A_k) &\leqslant C n \cdot \frac{1}{\eta} \rme^{-\frac{(k-1)^2}{2\eta^2}} \leqslant \frac{C }{\eta^2} \rme^{-\frac{(k-1)^2}{2\eta^2}}, 
				\quad \forall \  2 \leqslant k \leqslant \lceil \rme^{\frac1{2\eta}} \rceil.
		\end{aligned}
		\]
		It follows that for any $p\geqslant 1$, there is some positive constant $C_{p}^{\prime}$ not depending on $\eta$ such that
		\begin{equation}
			\label{ineq:E_bound of A}
			\begin{aligned}
			\sum_{k=1}^{k_0} \bbE\left[ \sup_{0 \leqslant m \leqslant n}  \abs{Y_m}^p \mathbbm{1}_{A_k}\right] 
			&\leqslant \frac{C}{\eta^{3p/2}} \cdot \eta^{\frac{\abs{\log \eta}}{2}-1} +\frac{C }{\eta^2} \sum_{k=2}^{\infty} \frac{k^p}{\eta^{3p/2}} 
			\rme^{-\frac{(k-1)^2}{2\eta^2}} 
			 \\
			&\leqslant C + 2^p C \sum_{k=1}^{\infty} \left( \frac{k}{\eta} \right)^{\frac{3p}{2}+2} \rme^{-\frac12\left(\frac{k}{\eta}\right)^2} \\
			&\leqslant C + 2^p C \int_{\lfloor \frac{1}{\eta} \rfloor -1}^{\infty} y^{\frac{3p}{2}+2} \rme^{-\frac{y^2}{2}} \dd y \leqslant  C_{p}^{\prime},
			\end{aligned}
		\end{equation}
		Here the second inequality holds since one can find $\eta\leqslant \rme^{-(3p+2)}$ such that $\frac{\abs{\log n}}{2}- \frac{3p+2}{2} \geqslant 0$ 
		and the fact $(k+1)^p \leqslant (2k)^p$ for any $k\geqslant1$ and any fixed $p\geqslant1$. Besides, $\int_z^{\infty} z^q \rme^{-z^2/2} \dd z < \infty$ 
		for any $q,z>0$ yields the last inequality. 
		
		On the event $B_k$,  due to EM scheme \eqref{EM scheme BM} and $\abs{Y_0} = \abs{x_0} < \eta^{-\frac12 }$, we have
		\[
			\abs{Y_1} \leqslant \abs{Y_0} ( 1+ \eta \log(1 + \abs{Y_0} ) ) + \sqrt{\eta} \abs{N_1} \leqslant \frac{2(k+1)}{\sqrt{\eta}} \rme^{\frac1{2\eta}},
		\]
		where the last inequality comes from $\eta^{-\frac12} + 
		\sqrt{\eta} \log(1+ \eta^{-\frac12}) \leqslant \eta^{-\frac12} \rme^{\frac1{2\eta}}$ as $\eta \leqslant \frac12$. Furthermore, for $\abs{Y_2}$ we have
			\begin{align*}
				\abs{Y_2} &\leqslant \abs{Y_1} ( 1+ \eta \log( 1 + \abs{Y_1} ) ) + \sqrt{\eta} \abs{N_2} \\
				&\leqslant \frac{2(k+1)}{\sqrt{\eta}} ( 1+k ) \rme^{\frac1{2\eta}} + \frac{(k+1)}{\sqrt{\eta}} \rme^{\frac1{2\eta}} \\
				&\leqslant \left[ 2(k+1)^2 + (k+1) \right] \frac{1}{\sqrt{\eta}} \rme^{\frac1{2\eta}}
			\end{align*}
		Hence, by induction, we  obtain that for each $k \geqslant 1$,
		\begin{equation}
			\label{ineq: bound on B_k}
			\abs{Y_m} \leqslant \left[ 2(k+1)^m + (k+1)^{m-1} + \cdots + (k+1) \right] \frac{1}{\sqrt{\eta}} \rme^{\frac1{2\eta}} := a_m \frac{1}{\sqrt{\eta}} \rme^{\frac1{2\eta}} .
		\end{equation}
		Hence, due to $a_m \leqslant 2m(k+1)^m$ and $\eta=O(\frac{1}{n})$, we have
		\[
			\sup_{1\leqslant m \leqslant n} \abs{ Y_m } \leqslant \frac{2C}{\eta^{3/2}} (k+1)^{\frac{T}{\eta}} \rme^{\frac1{2\eta}}.
		\]
		Besides, applying Lemma \ref{lem:bound of BM} again, 
		\[
			\mathbb{P}(B_k) \leqslant \frac{C }{\eta} \rme^{-\frac12 \left( \frac{k}{\eta} \rme^{1/(2\eta)} \right)^2} , \quad \forall \ k \geqslant 1.
		\]
		Due to similar argument \eqref{ineq:E_bound of A} for any $p\geqslant 1$, there  exists a constant  $C_{p}^{\prime\prime}$ not depending on $\eta$ such that 
		\begin{equation}
			\label{ineq:E_bound on B}
			\begin{aligned}
				\sum_{k=1}^{\infty} \bbE\left[ \sup_{0 \leqslant m \leqslant n} \! \abs{Y_m}^p \mathbbm{1}_{B_k}\right] \!\!
				\leqslant \! \frac{2^p C^{p+1}}{\eta^{(3p+2)/2}} \sum_{k=1}^{\infty} { (k+1)^{\frac{pC}{\eta}}} \rme^{\frac{p}{2\eta}} \cdot  
				\rme^{-\frac12 \left( \frac{k}{\eta} \rme^{1/(2\eta)} \right)^2}
				\leqslant \! C_{p}^{\prime\prime}.
			\end{aligned}
		\end{equation}	
		
		In sum, equations \eqref{ineq:E_bound of A} and \eqref{ineq:E_bound on B}yield that there exists a constant $C_p$ not depending on $\eta$ such that
		\[
			\mathbb{E} \left[\sup_{1 \leqslant m \leqslant n}|Y_m|^p  \mathbbm{1}_{A^c_0}\right] \leqslant 
			\sum_{k=1}^{k_0} \bbE\left[ \sup_{0 \leqslant m \leqslant n} \! \abs{Y_m}^p \mathbbm{1}_{A_k}\right]  + 
			\sum_{k=1}^{\infty} \bbE\left[ \sup_{0 \leqslant m \leqslant n} \! \abs{Y_m}^p \mathbbm{1}_{B_k}\right] 
			\leqslant C_p \ ,
		\]
		that is, \eqref{e:XlogX2} holds. And we complete the proof.
	\end{proof}

%
\section{EM scheme in the case of $\alpha \in (0,2)$}	\label{sec 3}	

		In this section, we prove Theorems \ref{thm 2} and \ref{thm 3}. 

		\begin{proof}[Proof of Theorem \ref{thm 2}] 
			\label{proof of thm 2}
			We only prove Part $(\runum{2})$ of the theorem in detail, since Part $(\runum{1})$ can be shown in a similar way 
			by replacing the estimates in Lemma \ref{Lemma 2.1} with the ones in Lemma \ref{Lemma 2.2}.
			
			Recall the formulations of \eqref{e:EM scheme}:
		\[
			\widetilde{Y}_{k+1} =\widetilde{Y}_{k} - \eta \widetilde{Y}_{k} \log \left( 1+ \big| \widetilde{Y}_{k} \big| \right) 
			+ \frac{1}{\sigma} \eta^{1/\alpha} \widetilde{Z}_{k+1} , \quad \widetilde{Y}_{0}:= x_0 ,
		\]
		where $\eta = \frac{T}{n}$ for all $k=,0,1,\dotsc,n-1$, and the definition of $K_2$ in the theorem:
			\[
			K_2 = \frac2T \left( \rme^{\kappa_\alpha/\beta} + \frac2\sigma  \right),
			\]
			for any fixed time $T > 0$, where $\kappa_\alpha$ is given in Lemma \ref{Lemma 2.1}. Due to $\eta=\frac{T}n \leqslant 1$,
			 it is easy to check that
			\begin{equation}
				\label{ine:K_2}
				\eta \cdot nK_2 - \frac{2(1+\eta)}{\sigma} \eta^{1/\alpha}  -1 
				= 2 \rme^{\kappa_\alpha/\beta} - 1  + \frac4\sigma ( 1 -  \eta^{1/\alpha})
				\geqslant \rme^{\kappa_\alpha/\beta}.
			\end{equation}
			We will apply this equation repeatedly.  Consider a sequence of events  $\widetilde{\Omega}_n \subseteq \Omega$, $n\in \mathbb{N}$, defined by
			\begin{equation}
				\begin{aligned}
					\widetilde{\Omega}_n := \Bigg\{ \omega \in \Omega \  \bigg| \ 
					&\big| \widetilde{Z}_k(\omega) \big| \in \left[ 1+\eta, 2+2\eta \right], \ \forall k = 2, \cdots, n\ ;\\
					&\quad \frac{\eta^{1/\alpha}}{\sigma}  \big|\widetilde{Z}_{1}(\omega) \big| \geqslant \abs{x_0}\big(1+\log(1+\abs{x_0}) \big) +  \exp\{nK_2\}
					\Bigg\}.
				\end{aligned}
			\end{equation}
			For all $\omega \in \widetilde{\Omega}_n$, we verify by induction that the following holds:
			\begin{equation}
				\label{claim1}
				\big| {\widetilde{Y}_m(\omega)} \big| \geqslant \exp\left\{ \frac{\kappa_\alpha}{\beta} (m-1) + n K_2 \right\}, \quad \forall\ m= 1,\dotsc,n.
			\end{equation}
			If $m=1$, 
			the triangle inequality yields that
			\begin{equation*}
				\begin{aligned}
					&\mathrel{\phantom{=}}
					\big|{\widetilde{Y}_1(\omega)}\big| = \big| { x_0 - x_0\log(1+\abs{x_0}) + \frac{\eta^{1/\alpha}}{\sigma} \widetilde{Z}_1(\omega) } \big| \\
					&\geqslant \frac{\eta^{1/\alpha}}{\sigma}  \big|{\widetilde{Z}_{1}(\omega)} \big|
					- \abs{ x_0 } \big( 1 +  \log( 1 + \abs{x_0} )  \big)
					\geqslant \exp\{ nK_2 \} .
				\end{aligned}			
			\end{equation*} 	
When $m=2$, we have
			\begin{align*}
				\big| {\widetilde{Y}_2(\omega)} \big|  &\geqslant \eta \big| {\widetilde{Y}_1(\omega)}\big|  \log\left( 1+\big| {\widetilde{Y}_1(\omega)} \big|  \right) - \frac{\eta^{1/\alpha}}{\sigma} \big| {\widetilde{Z}_2(\omega)} \big|  - \big| {\widetilde{Y}_1(\omega)} \big|  \\
				&\geqslant \big|  {\widetilde{Y}_1(\omega)} \big|  \cdot \left[ \eta \log \left( \big| {\widetilde{Y}_1(\omega)} \big| \right)  - \frac{\eta^{1/\alpha}}{\sigma}  \big| { \widetilde{Z}_2(\omega) } \big|  - 1 \right] \\ 
				&\geqslant 
				\exp\{ n K_2 \} \cdot \left[ \eta \cdot nK_2 - \frac{2(1+\eta)}{\sigma} \eta^{1/\alpha} -1 \right]
				\geqslant  \exp\left\{ \frac{\kappa_\alpha}{\beta} + nK_2 \right\} .
			\end{align*}             
			The second inequality is a consequence to $\big| \widetilde{Y}_1(\omega) \big| \geqslant 1$. And the third inequality arises from 
			$\big| \widetilde{Y}_1(\omega) \big| \geqslant \rme^{nK_2}$ and $\big| \widetilde{Z}_2(\omega) \big| \leqslant 2(1+\eta)$ for all 
			$\omega\in \widetilde{\Omega}_n$. The last inequality follows from  \eqref{ine:K_2}. Therefore, \eqref{claim1} holds for $m=2$. 
			
			For the induction step $k\to k+1$, we assume that \eqref{claim1} holds for $\ell=1,\cdots, k$. In particular, $\big| \widetilde{Y}_{\ell}(\omega) \big| 
			\geqslant \exp\{ nK_2 \} \geqslant 1$ for all $\ell=1,\cdots, k$ and all $\omega\in \widetilde{\Omega}_n$. Analogous with the case of $m=2$, 
			we have 
			\begin{align*}
				\big| {\widetilde{Y}_{k+1}(\omega)} \big| &\geqslant \eta \big| {\widetilde{Y}_k(\omega)} \big|  \log\left( 1+\big| {\widetilde{Y}_k(\omega)} \big|  \right) - \frac{\eta^{1/\alpha}}{\sigma} \big| {\widetilde{Z}_{k+1}(\omega)} \big|  - \big| {\widetilde{Y}_k(\omega)} \big|  \\
				&\geqslant \big| {\widetilde{Y}_k(\omega)} \big|  \cdot \left[ \eta \log \left( \big| {\widetilde{Y}_k(\omega)} \big|  \right)  -
				 \frac{\eta^{1/\alpha}}{\sigma}  \big| { \widetilde{Z}_{k+1}(\omega) } \big|  - 1 \right] \\ 
				&\geqslant 
				\exp\left\{ \frac{\kappa_\alpha}{\beta}(k-1) + n K_2 \right\} \cdot \left[ \eta \cdot nK_2 - \frac{2\eta^{1/\alpha}}{\sigma}(1+\eta) -1 \right] \\
				&\geqslant  \exp\left\{ \frac{\kappa_\alpha}{\beta}k + nK_2 \right\} .
			\end{align*}
			The third inequality arises from $\big| \widetilde{Y}_k(\omega) \big| \geqslant \exp\{\frac{\kappa_\alpha}{\beta}(k-1)+nK_2\}$ and 
			$\big| \widetilde{Z}_{k+1}(\omega) \big| \leqslant 2(1+\eta)$ for all $\omega\in \widetilde{\Omega}_n$. The last inequality follows from inequality 
			\eqref{ine:K_2}.
			Therefore, claim \eqref{claim1} holds. In particular, for all $ n \in \mathbb{N},$
			\begin{equation}
				\label{ine:bound of Y_n}
				\big| {\widetilde{Y}_{n}(\omega)} \big|  \geqslant \exp\left\{ (n-1) \frac{\kappa_\alpha}{\alpha} + nK_2 \right\} , \quad \forall\ \omega 
				\in \widetilde{\Omega}_n.			
			\end{equation}
			
			Next, we establish a lower bound for the probability of $\widetilde{\Omega}_n$ by using Lemma \ref{Lemma 2.1}. Firstly, we have
			\begin{equation*}
				\begin{aligned}
					&\mathrel{\phantom{=}}
					\mathbb{P} \left( \frac{ \eta^{\frac{1}{\alpha}} }{\sigma}  \abs{\widetilde{Z}_{1}(\omega)} \geqslant \abs{x_0}\big(1+\log(1+\abs{x_0})\big)  
					+ \exp\left\{ n K_2 \right\} \right)  \\
					&\geqslant \frac{\eta }{\sigma^\alpha} \cdot \Big( \abs{x_0}\big(1+\log(1+\abs{x_0}) \big) + \exp\left\{
					n K_2 \right\} \Big)^{-\alpha} \\
					&\geqslant \frac{T}{2\sigma^\alpha} \cdot \left( \frac{ \abs{x_0}^{\alpha}\big(1+\log(1+\abs{x_0}) \big)^{\alpha} }{ \exp\{ \alpha n K_2 \} } 
					+ 1 \right)^{-1} \cdot \frac1n \exp\{ -\alpha n K_2 \} ,
				\end{aligned}
			\end{equation*}
			where the last inequality comes from $(a+b)^{\alpha} \leqslant 2(a^\alpha+b^\alpha)$ for all $a, b >0$ and $\alpha \in (0,2)$. Besides, 
			Lemma \ref{Lemma 2.1} also yields that for all $k=2,\dotsc,n$, 
			\[
			\mathbb{P}\left( \big| { \widetilde{Z}_k(\omega) } \big|  \in \left[ 1+\eta, 2+2\eta \right] \right) \geqslant \frac1{\kappa_\alpha} \left( 1+\eta \right)^{-\alpha} .
			\]
			Thus, combining condition $\rme^{n K_2} \geqslant \abs{x_0} \big( 1 + \log(1+\abs{x_0}) \big)$ for sufficiently large $n$, 
			\begin{equation}\label{ine:low of P}
				\begin{aligned}
					\mathbb{P}\big(\widetilde{\Omega}_n\big)
					&\geqslant \frac{T}{ 4 \sigma^{\alpha}  } \cdot \frac1{n} \exp \left\{ -\alpha n K_2 \right\} \cdot \frac1{\kappa_\alpha^n}
					\left( 1+\frac{T}{n} \right)^{-\alpha n}
					\geqslant C \cdot \frac{\rme^{ -\alpha n K_2 } }{n \kappa_\alpha^n}
				\end{aligned}
			\end{equation}
			holds for some constant $C>0$ independent of $n$. Hence, combining equation \eqref{ine:bound of Y_n} with \eqref{ine:low of P} leads to
			\begin{equation*}
				\begin{aligned}
					\mathbb{E} \left[ \big| \widetilde{Y}_n \big|^{\beta}\right] &\geqslant  
					\mathbb{E} \left[ \big| \widetilde{Y}_n \big|^{\beta} \mathbbm{1}_{\widetilde{\Omega}_n} \right] 
					\geqslant \mathbb{P}\big(\widetilde{\Omega}_n\big) \cdot \exp\big\{ \kappa_\alpha(n-1)  + \beta n K_2 \big\} \\
					&\geqslant  \frac{C}{n} \cdot 
					\left( \frac{\exp\{\beta K_2 + \kappa_\alpha -\alpha K_2 \}}{  \kappa_{\alpha}  } \right)^{n}.
				\end{aligned}
			\end{equation*}
			Condition $K_2 < (\kappa_\alpha-\log \kappa_\alpha)/(\alpha - \beta)$ implies $\rme^{\beta K_2 + \kappa_\alpha - \alpha K_2} > \kappa_\alpha$.  
			Consequently, we can obtain that
			\begin{equation}
				\lim\limits_{n\to \infty}\mathbb{E}  \big|\widetilde{Y}_n\big|^{\beta}  = \infty, \quad \forall \ \beta \in (0, \alpha).
			\end{equation}
The proof is completed.
		\end{proof}

	Next, we prove Theorem \ref{thm 3}.
		 
		\begin{proof}[Proof of Theorem \ref{thm 3}]
		Again, we only provide a proof for Part (ii) and omit details for proving Part (i).
		Recall \eqref{e:general EM} and that that functions $f$ and $g$ satisfy Assumption \ref{assump_f_g}.	
			For all $n\in\mathbb{N}$, define $r_n$ as
			\begin{equation}\label{e:define r_N}
				\begin{aligned}
					r_n:=\max \Bigg\{ 2, H, & \left(\frac{4H }{\eta}+\frac{4 H^2}{\eta \sigma} \big(1+\eta \big) \eta^{1/\alpha}\right)^{\frac{1}{\gamma - \lambda}}, \\
				&\qquad \left( \sigma H \big(2+H \eta \big) \big(1+\eta\big)^{-1}\eta^{-\frac1\alpha} \right)^{\frac{1}{\gamma-\lambda}}
				 \Bigg\}\in [2,\infty) .
				\end{aligned}
			\end{equation}
			where $\eta = \frac{T}{n}$. The third term of the right hand side  ensures that 
			\begin{equation}\label{e:ineq of r_N}
				\frac{\eta}{2H} r_n^{\gamma-\lambda} 
				\geqslant 2+\frac{2H}{\sigma} \left(1+\eta\right)\eta^{\frac1\alpha},
			\end{equation}
			and the last term  ensures 
			\begin{equation}
				\label{e:ineq of r_N (0,1)}
				\frac{1}{\sigma H} ( 1+\eta ) \eta^{\frac1\alpha}  r_n^{\gamma - \lambda} \geqslant 2 + H \eta.
			\end{equation}
			Both equations will be used below. 
			
			 Since $g(x_0) \neq 0$, there exists a constant $M\geqslant 1$ such that $\abs{g(x_0) }\geqslant M^{-1}$ and $\abs{x_0}+T \abs{f(x_0)} \leqslant M$. 
			 And we consider the events $\widetilde{\Omega}_n \subseteq \Omega$ for all $n\in \mathbb{N}$ defined as
			\begin{equation}\label{e:definition of Omega_N}
				\begin{aligned}
					\widetilde{\Omega}_{n}:=\bigg\{ \omega\in \Omega \ \Big| \ &  \big|  \widetilde{Z}_k(\omega)\big| 
					\in \big[1+\eta , 2+2\eta \big],\forall k=2,\dotsc,n ;  \\
					& \quad \qquad \frac{\eta^{1/\alpha}}{\sigma}\big| \widetilde{Z}_1(\omega)\big| \geqslant M(r_n+M) 
					\bigg\} .
				\end{aligned}
			\end{equation}
			
			We claim that for every  $\omega\in \widetilde{\Omega}_n$ and $ m\in \left\{1,2,\cdots,n\right\}$,
			\begin{equation}\label{e:claim of Y_N}
				|\widetilde{Y}_m(\omega)| \geqslant r_n^{\lambda^{(m-1)}}. 
			\end{equation} 
			By induction, in the base case $m=1$, the triangle inequality leads to 
			\begin{align*}
				\big|\widetilde{Y}_1(\omega)\big|=&\big|x_0+ \eta f(x_0) +\frac{\eta^{1/\alpha}}{\sigma}  g(x) \widetilde{Z}_1(\omega)\big|\\
				\geqslant & \frac{\eta^{1/\alpha}}{\sigma}  \abs{g(x_0)} \big|{\widetilde{Z}_1(\omega)} \big| -|x_0|- \eta \abs{ f(x_0) }\\
				\geqslant &\frac{\eta^{1/\alpha}}{M  \sigma }  \big|{\widetilde{Z}_1 (\omega)}\big| - M \geqslant \frac{M(r_n+M)}{M}-M\geqslant r_n ,
			\end{align*}
			which follows from the definition \eqref{e:general EM} of $\widetilde{Y}_1$ and \eqref{e:definition of Omega_N} of $\widetilde{\Omega}_{n}$. 
			For the induction step $m\to m+1$, We assume that equation \eqref{e:claim of Y_N} holds for $k\in \left\{1,2,\cdots,m\right\}$. In particular, 
			we can obtain $|\widetilde{Y}_k(\omega)| \geqslant r_n \geqslant H\geqslant 1$. Additionally, the EM scheme \eqref{e:general EM} yields that
			\begin{equation}
			\label{ine1}
			\begin{aligned}
		 \big|\widetilde{Y}_{m+1}(\omega)\big|=&\Big|\widetilde{Y}_m(\omega)+ \eta  f\big(\widetilde{Y}_m(\omega)\big) +\frac{\eta^{1/\alpha}}{\sigma}  
		 g \big(\widetilde{Y}_m(\omega) \big) \widetilde{Z}_{m+1}(\omega)\Big|\\
			\geqslant & \Big| \eta f\big(\widetilde{Y}_m(\omega)\big) +\frac{\eta^{\frac1\alpha}}{ \sigma }  g \big(\widetilde{Y}_m(\omega)\big)
			 \widetilde{Z}_{m+1}(\omega)\Big|-\big|\widetilde{Y}_m(\omega)\big|\\
			\geqslant &\max\left\{ \eta \left|f\big(\widetilde{Y}_m(\omega)\big)\right|,\left|g\big(\widetilde{Y}_m(\omega)\big)\right|\left|\frac{\eta^{1/\alpha}}{ \sigma } 
			\widetilde{Z}_{m+1}(\omega)\right|\right\}\\
			&-\min\left\{ \eta \left|f\big(\widetilde{Y}_m(\omega)\big)\right|,\left|g\big(\widetilde{Y}_m(\omega)\big)\right|\left|\frac{\eta^{1/\alpha}}{ \sigma }  
			\widetilde{Z}_{m+1}(\omega)\right|\right\} 	-\big|\widetilde{Y}_m(\omega) \big|,
				\end{aligned}
			\end{equation}	
	where we have repeatedly used the triangle inequality. Since $\big| \widetilde{Z}_{m+1}(\omega) \big| \in [1+\eta, 2+2\eta]$ for all  
	$ \omega \in \widetilde{\Omega}_n$, we notice that
			\begin{equation}
			\label{max}
			\begin{aligned}
			&\mathrel{\phantom{\ge}}
			\max\left\{ \eta \left|f\big(\widetilde{Y}_m(\omega)\big)\right|,\left|g\big(\widetilde{Y}_m(\omega)\big)\right|\left|\frac{\eta^{1/\alpha}}{ \sigma } 
			\widetilde{Z}_{m+1}(\omega)\right|\right\} \\
			&\geqslant \max\left\{ \eta \left|f\big(\widetilde{Y}_m(\omega)\big)\right|, \frac{(1+\eta)\eta^{1/\alpha}}{ \sigma }
			\left|g\big(\widetilde{Y}_m(\omega)\big)\right| \right\},
				\end{aligned}
			\end{equation}	
			and	
			\begin{equation}
				\label{min}
				\begin{aligned}
			&\mathrel{\phantom{\ge}}
					\min\left\{ \eta \left|f\big(\widetilde{Y}_m(\omega)\big)\right|,\left|g\big(\widetilde{Y}_m(\omega)\big)\right|\left|\frac{\eta^{1/\alpha}}{ \sigma } 
					\widetilde{Z}_{m+1}(\omega)\right|\right\} \\
					&\leqslant \min\left\{ \eta \left|f\big(\widetilde{Y}_m(\omega)\big)\right|, \frac{2(1+\eta)\eta^{1/\alpha}}{ \sigma }
					\left|g\big(\widetilde{Y}_m(\omega)\big)\right| \right\},
				\end{aligned}
			\end{equation}	
			If $\alpha \in (1,2)$, we have $\frac\eta2 \leqslant \frac{1}{ \sigma } (1+\eta) \eta^{1/\alpha}$  for sufficiently large $n$. By the definition of $\sigma$, 
			we know that $\sigma \leqslant 2$ as $\alpha = 1$. Thus $\frac\eta2 \leqslant \frac1\sigma (1+\eta) \eta$ if $\alpha=1$. Consequently, it follows from
			 \eqref{ine1}, \eqref{max} and \eqref{min} that for $\alpha \in [1,2)$, 
	\begin{align*}
			\big|\widetilde{Y}_{m+1} (\omega)\big| 
			\geqslant & \frac\eta2 \max\left\{ \left| f\big(\widetilde{Y}_m(\omega) \big)\right|,\left| g \big(\widetilde{Y}_m(\omega) \big)\right| \right\} \\
			&-\frac{2(1+\eta) \eta^{1/\alpha}}{ \sigma } \min \left\{ \left| f\big(\widetilde{Y}_m(\omega) \big)\right|,\left| g \big(\widetilde{Y}_m(\omega) \big)\right| \right\} 
			-\big|\widetilde{Y}_m(\omega)\big|\\
			\geqslant & \frac{\eta}{2H} \big|\widetilde{Y}_m(\omega)\big|^{\gamma} -\frac{2H}{ \sigma } (1+\eta)\eta^{\frac1\alpha} 
			\big|\widetilde{Y}_m(\omega)\big|^{\lambda}-\big|\widetilde{Y}_m(\omega)\big|^{\lambda}\\
			= &  \big|\widetilde{Y}_m(\omega)\big|^{\lambda} \left[\frac{\eta}{2H} \big|\widetilde{Y}_m (\omega)\big|^{\gamma-\lambda} -
			\frac{2H}{ \sigma } (1+\eta)\eta^{\frac1\alpha}-1\right]\\
			\geqslant & \big|\widetilde{Y}_m(\omega)\big|^{\lambda} \left[\frac{\eta}{2H} \, r_n^{\gamma-\lambda}-\frac{2H}{ \sigma } (1+\eta)\eta^{\frac1\alpha}-1\right]\\
			\geqslant & \big|\widetilde{Y}_m(\omega)\big|^{\lambda},
	\end{align*}
where the first inequality comes from Assumption \ref{assump_f_g} and the last inequality follows from inequality \eqref{e:ineq of r_N}. 
On that other hand, in the case of $\alpha\in(0,1)$, we have $\frac{2}{ \sigma } (1+\eta) \eta^{1/\alpha} \leqslant \eta$ for $n$ large enough. 
Then, a similar argument leads to
			\[
				\begin{aligned}
					\big|\widetilde{Y}_{m+1} (\omega)\big| 
					&\geqslant \big|\widetilde{Y}_m(\omega)\big|^{\lambda} \left[\frac{(1+\eta)\eta^{1/\alpha}}{\sigma H } 
					r_n^{\gamma-\lambda}-H \eta-1\right] \geqslant \big|\widetilde{Y}_m(\omega)\big|^{\lambda} ,
				\end{aligned}
			\]
where the last inequality follows from \eqref{e:ineq of r_N (0,1)}. Hence, the induction hypothesis yields that for all $\alpha \in (0,2)$
			\begin{equation*}
				\big|\widetilde{Y}_{m+1}(\omega)\big| \geqslant
				\big|\widetilde{Y}_m(\omega)\big|^{\lambda} \geqslant
				\left(r_n^{ \lambda^{m-1} }\right)^{\lambda}= r_n^{\lambda^m}.
			\end{equation*}
This proves the claim (\ref{e:claim of Y_N}). In particular, since $r_n \geqslant 2$, we obtain
	\begin{equation}\label{e:lower bound of Y_N}
				\big|\widetilde{Y}_n(\omega)\big| \geqslant  r_n^{ \lambda^{n-1} } 
				\geqslant 2^{ \lambda^{n-1}} , \quad \forall \ \omega\in \widetilde{\Omega}_n, \ \text{and}\  n\in\mathbb{N} 
			\end{equation}
						
Furthermore, by applying Lemma \ref{Lemma 2.1}, we derive the following lower bound for the probability of $\widetilde{\Omega}_n$,
			\begin{equation*}
				\begin{aligned}
					\mathbb{P} (\widetilde{\Omega}_n) =&\mathbb{P}\left( \left|\frac{\eta^{1/\alpha}}{ \sigma } \widetilde{Z}_1\right| 
					\geqslant M (r_n+M)\right) \prod_{k=2}^{n}
					\mathbb{P} \left( \left|\widetilde{Z}_{k}\right| \in \big[1+\eta,2+2\eta \big]\right)\\
					\geqslant &  \mathbb{P} \left( \big|\widetilde{Z}_1\big| \geqslant \frac{ \sigma  M(r_n+M)}{\eta^{1/\alpha}}\right) \!\! 
					\left[\mathbb{P} \left( \big|\widetilde{Z}_{1}\big| \in \big[1+\eta,2+2\eta\big]\right)\right]^{n} \\
					\geqslant & 
					\frac{T}{n \kappa_\alpha^n} \left[\frac{1}{ \sigma  M(r_n+M)}\right]^{\alpha} \left(1+\frac{T}{n}\right)^{-\alpha n}.
				\end{aligned}
			\end{equation*}
Thus, there exists a constant $c\in(1,\infty)$ such that
			\begin{equation}\label{e:lower bound of P(Omega_N)}
				\mathbb{P} (\widetilde{\Omega}_n) \geqslant cn^{-c} \kappa_\alpha^{-n}
			\end{equation}
			for all sufficiently large $n$. 
			
Combining equations (\ref{e:lower bound of Y_N}) with (\ref{e:lower bound of P(Omega_N)}) gives that for any $\beta\in(0,\alpha)$
			\begin{equation}
				\begin{aligned}
					\lim\limits_{n\to \infty}\mathbb{E} \left[ \big|\widetilde{Y}_n\big|^{\beta}\right] 
					&\geqslant \lim\limits_{n\to \infty} \mathbb{E} \left[ \big|\widetilde{Y}_n\big|^{\beta} \mathbbm{1}_{\widetilde{\Omega}_n} \right] 
					\geqslant 
					\lim\limits_{n\to \infty} \left[\mathbb{P}[\widetilde{\Omega}_n] \, r_n^{\beta \lambda^{(n-1)}}\right]\\
					& \geqslant \lim\limits_{n\to\infty} \big(cn^{-c} \kappa_\alpha^{-n} \big) \cdot 2^{\beta \lambda^{n-1}}=\infty.
				\end{aligned}
			\end{equation}
			
	For proving Part $(\runum{1})$ of Theorem \ref{thm 3}, we can employ a similar approach to the EM scheme . 
{Define the parameter $r_n$ as in  \eqref{e:define r_N} but without including the term $\sigma$.}
	We consider a sequence of events $\Omega_n$ defined by 
	\[
	\begin{aligned}
	{\Omega}_{n}:=\bigg\{ \omega\in \Omega \ \Big| \ &  \big|  {Z}_k(\omega)\big| \in \big[1+\eta , 2+2\eta \big],\forall k=2,\dotsc,n ;  \\
					& \quad \qquad {\eta^{1/\alpha}}\big| {Z}_1(\omega)\big| \geqslant M(r_n+M) 
					\bigg\} .
				\end{aligned}
			\]
	Then, an argument similar to that for Part (ii) and Lemma \ref{Lemma 2.2} yield the desired conclusion. The proof is completed. 
		\end{proof}
		
\section{Simulations }
\label{sec:simulation}
	
	In this section, we present some numerical simulations that illustrate the convergence and divergence of EM scheme for $d=1$.
	
	{Firstly, we consider SDE \eqref{eq:critical sde} with $\alpha =2$ and the corresponding EM scheme \eqref{EM scheme BM}. We set the time interval $t\in[0,T]$ with $T = 10$ and $100$. For both time intervals, we take initial values $Y_0$ of $1$, $5$ and $10$ respectively. And set step size $\eta$ to $0.001$ as $T=10$, and to $0.01$ as $T=100$. Besides, we set $n = \frac{T}{\eta} = 10\,000$ for each case. More precisely, Figures \ref{figure 1} and \ref{figure 2} below illustrate the simulations of the EM scheme 
	for the second absolute moment $\bbE \abs{Y_k}^2$ over the range $0\leqslant k \leqslant n$ with iteration steps $n = 10000$ and initial values of $Y_0 = 1$, $5$, and $10$. In each figure, the blue line corresponds to $Y_0 = 1$, the green to $Y_0 = 5$, and the red to $Y_0 = 10 $.}
	
	
	\begin{figure}[htbp]
		\centering
		\includegraphics[width=1.0\linewidth]{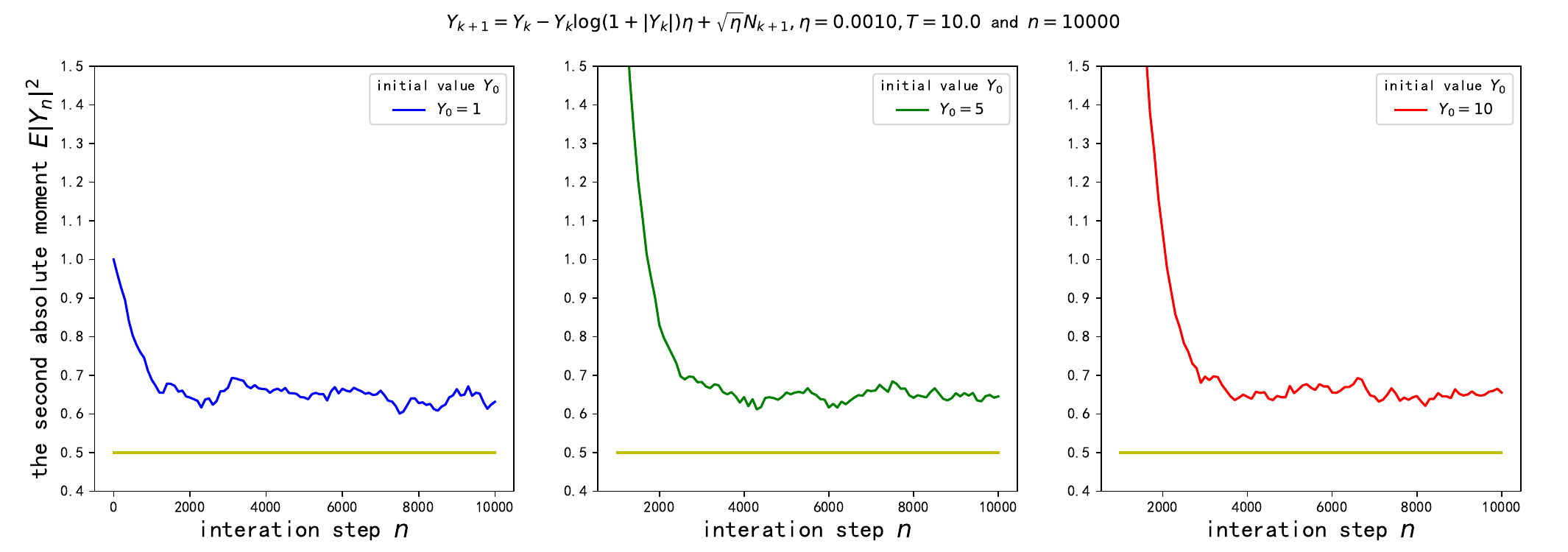}
		\caption{As $T = 10$, simulations values of the second absolute moment $\bbE\abs{Y_k}^2$ for the EM scheme \eqref{EM scheme BM} with initial $Y_0 = 1, 5, 10$, $\eta = 0.001 $ and iteration steps $n = 10\,000$, $0\leqslant k \leqslant n$. } \label{figure 1}
	\end{figure}
	\begin{figure}[htbp]
		\centering
		\includegraphics[width=1.0\linewidth]{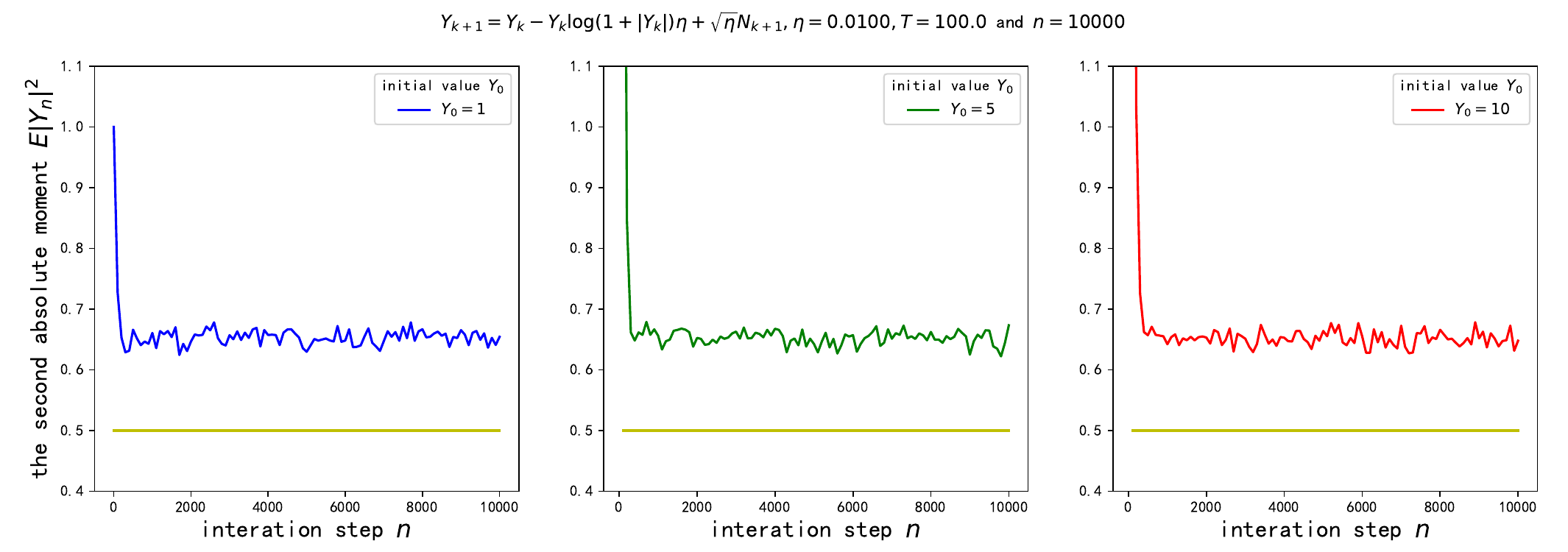}
		\caption{As $T = 100$, simulations values of the second absolute moment $\bbE\abs{Y_k}^2$ for the EM scheme \eqref{EM scheme BM} with initial $Y_0 = 1, 5, 10$, $\eta = 0. 01 $ and iteration steps $n = 10\,000$, $0\leqslant k \leqslant n$. } \label{figure 2}
	\end{figure}

	{Figures \ref{figure 1} and \ref{figure 2} indicate that $\{\bbE \abs{Y_k}^2, k \leqslant n\}$ are bounded and has a clear decreasing trend with respect to $0\leqslant k \leqslant n$ for each initial value and step size.}
	
	For SDE \eqref{eq:critical sde} with $\alpha \in (0,2)$ and the corresponding EM scheme \eqref{e:EM scheme}, 
	due to the condition of $K_2$ in the proof of  Theorem \ref{thm 2}, we choose $T=100$ here. 
	And we consider three cases, that is, $\alpha = 0.5, 1.0$, and $1.5$. For each case, we let $\beta$ be $\frac{\alpha}{8}$, $\frac{\alpha}{4}$ and $\frac{\alpha}{2}$. 
	The simulated values of the $\beta$-th moment of $\widetilde{Y}_n$ in these cases are listed in Tables \ref{table alpha 0.5}, \ref{table alpha 1.0}, and \ref{table alpha 1.5}.
		
		\begin{table}[thp]
			\begin{center}
		\caption{Simulated values of the absolute moment for the EM scheme \eqref{e:EM scheme} with $T=100$, $\alpha = 0.50$ and  $n=\{100,105,110,\dotsc,145\}$. }
		\label{table alpha 0.5}
				\begin{tabular}{p{2.0cm}<{\centering}| p{3cm}<{\centering}| p{3cm}<{\centering}| p{3cm}<{\centering}  } 
				\toprule 
	\multirow{2}{*}{$n$} & 
	\multicolumn{1}{c |}{ \multirow{2}{*}{$\bbE|{\widetilde{Y}_n}|^{{\alpha}/{8}}$} } & 
	\multicolumn{1}{c |}{ \multirow{2}{*}{$\bbE|{\widetilde{Y}_n}|^{{\alpha}/{4}}$}} & \multicolumn{1}{c }{ \multirow{2}{*}{$\bbE|{\widetilde{Y}_n}|^{{\alpha}/{2}}$} }  \\	
	\multirow{1}{*}{} & 
	\multicolumn{1}{c |}{ \multirow{1}{*}{} } & 
	\multicolumn{1}{c |}{ \multirow{1}{*}{} } & 
	\multicolumn{1}{c }{ \multirow{1}{*}{} }  \\
				\midrule 
					 $100$ & $1.8 \times 10^{13}$ & $6.2 \times 10^{26}$ & $3.8 \times 10^{55}$   \\
					 $105$ & $6.5 \times 10^{13}$ & $1.3 \times 10^{28}$ & $4.7 \times 10^{57}$   \\
					 $110$ & $2.8 \times 10^{14}$ & $3.0 \times 10^{29}$ & $3.6 \times 10^{60}$   \\
					 $115$ & $1.2 \times 10^{15}$ & $5.4 \times 10^{30}$ & $1.6 \times 10^{63}$   \\
					 $120$ & $5.5 \times 10^{15}$ & $1.3 \times 10^{32}$ & $4.3 \times 10^{65}$   \\
					 $125$ & $2.4 \times 10^{16}$ & $2.1 \times 10^{33}$ & $1.7 \times 10^{68}$  \\
					 $130$ & $1.1 \times 10^{17}$ & $4.5 \times 10^{34}$ & $1.3 \times 10^{71}$  \\
					 $135$ & $1.5 \times 10^{18}$ & $\infty$ & $\infty$  \\
					 $140$ & $\infty$ & $\infty$ & $\infty$   \\
					 $145$ & $\infty$ & $\infty$ & $\infty$  \\			 
				\bottomrule 
				\end{tabular}
			\end{center}
		\end{table}
		
		\begin{table}[thp]
			\begin{center}
				\caption{Simulated values of the absolute moment for the EM scheme \eqref{e:EM scheme} with $T=100$, $\alpha = 1.0$ and  $n=\{100,105,110,\dotsc,145\}$. }
				\label{table alpha 1.0}
				\begin{tabular}{p{2.0cm}<{\centering}| p{3cm}<{\centering}| p{3cm}<{\centering}| p{3cm}<{\centering}  } 
					\toprule 
					\multirow{2}{*}{$n$} & 
					\multicolumn{1}{c |}{ \multirow{2}{*}{$\bbE|{\widetilde{Y}_n}|^{{\alpha}/{8}}$} } & 
					\multicolumn{1}{c |}{ \multirow{2}{*}{$\bbE|{\widetilde{Y}_n}|^{{\alpha}/{4}}$}} & \multicolumn{1}{c }{ \multirow{2}{*}{$\bbE|{\widetilde{Y}_n}|^{{\alpha}/{2}}$} }  \\	
					\multirow{1}{*}{} & 
					\multicolumn{1}{c |}{ \multirow{1}{*}{} } & 
					\multicolumn{1}{c |}{ \multirow{1}{*}{} } & 
					\multicolumn{1}{c }{ \multirow{1}{*}{} }  \\
					\midrule 
					$100$ & $3.8 \times 10^{25}$ & $7.1 \times 10^{52}$ & $2.2 \times 10^{109}$   \\
					$105$ & $7.7 \times 10^{26}$ & $1.5 \times 10^{55}$ & $2.3 \times 10^{112}$   \\
					$110$ & $1.3 \times 10^{28}$ & $1.1 \times 10^{58}$ & $5,4 \times 10^{117}$   \\
					$115$ & $2.7 \times 10^{29}$ & $2.2 \times 10^{60}$ & $2.7 \times 10^{124}$   \\
					$120$ & $4.6 \times 10^{30}$ & $1.3 \times 10^{63}$ & $6.3 \times 10^{128}$   \\
					$125$ & $9.2 \times 10^{31}$ & $3.9 \times 10^{65}$ & $1.2 \times 10^{135}$  \\
					$130$ & $1.7 \times 10^{33}$ & $2.9 \times 10^{68}$ & $2.1 \times 10^{141}$  \\
					$135$ & $2.9 \times 10^{34}$ & $3.7 \times 10^{71}$ & $1.9 \times 10^{145}$  \\
					$140$ & $5.9 \times 10^{35}$ & $1.8 \times 10^{73}$ & $\infty$   \\
					$145$ & $\infty$ & $\infty$ & $\infty$  \\			 
					\bottomrule 
				\end{tabular}
			\end{center}
		\end{table}
		
		\begin{table}[thp]
			\begin{center}
				\caption{ Simulated values of the absolute moment for the EM scheme \eqref{e:EM scheme} with $T=100$, $\alpha = 1.5$ 
				and  $n=\{100,105,110,\dotsc,145\}$. }
				\label{table alpha 1.5}
				\begin{tabular}{p{2.0cm}<{\centering}| p{3cm}<{\centering}| p{3cm}<{\centering}| p{3cm}<{\centering}  } 
					\toprule 
					\multirow{2}{*}{$n$} & 
					\multicolumn{1}{c |}{ \multirow{2}{*}{$\bbE|{\widetilde{Y}_n}|^{{\alpha}/{8}}$} } & 
					\multicolumn{1}{c |}{ \multirow{2}{*}{$\bbE|{\widetilde{Y}_n}|^{{\alpha}/{4}}$}} & \multicolumn{1}{c }{ \multirow{2}{*}{$\bbE|{\widetilde{Y}_n}|^{{\alpha}/{2}}$} }  \\	
					\multirow{1}{*}{} & 
					\multicolumn{1}{c |}{ \multirow{1}{*}{} } & 
					\multicolumn{1}{c |}{ \multirow{1}{*}{} } & 
					\multicolumn{1}{c }{ \multirow{1}{*}{} }  \\
					\midrule 
					$100$ & $8.7 \times 10^{37}$ & $8.4 \times 10^{77}$ & $7.0 \times 10^{158}$   \\
					$105$ & $5.8 \times 10^{39}$ & $5.6 \times 10^{83}$ & $7.8 \times 10^{168}$   \\
					$110$ & $3.9 \times 10^{41}$ & $3.7 \times 10^{86}$ & $6.9 \times 10^{174}$   \\
					$115$ & $2.7 \times 10^{43}$ & $1.4 \times 10^{90}$ & $2.1 \times 10^{182}$   \\
					$120$ & $2.9 \times 10^{45}$ & $6.0 \times 10^{93}$ & $4.6 \times 10^{192}$   \\
					$125$ & $1.9 \times 10^{47}$ & $2.4 \times 10^{97}$ & $1.8 \times 10^{200}$  \\
					$130$ & $3.3 \times 10^{49}$ & $4.8 \times 10^{101}$ & $8.7 \times 10^{207}$  \\
					$135$ & $4.7 \times 10^{51}$ & $7.2 \times 10^{104}$ & $7.1 \times 10^{215}$  \\
					$140$ & $2.9 \times 10^{53}$ & $2.5 \times 10^{111}$ & $9.2 \times 10^{219}$   \\
					$145$ & $\infty$ & $\infty$ & $\infty$  \\			 
					\bottomrule 
				\end{tabular}
			\end{center}
		\end{table}
		
		From Tables \ref{table alpha 0.5}, \ref{table alpha 1.0}, and \ref{table alpha 1.5}, we observe that the simulated values of the 
		absolute moments of $\widetilde{Y}_n$ increase with respect to number of iteration $n$.  In addition, they suggest that,
		for smaller values of $\alpha \in (0,1)$,  lager values of time $T$ should be chosen 
		due to constant $\kappa_\alpha = (2^\alpha /(2^\alpha-1))$ in Lemma \ref{Lemma 2.1}. Consequently, these tables indicate that
		the behavior of the absolute $\frac{\alpha}{2}$-th moment varies with respected to $\alpha$ if we choose a fixed time $T$.
	
	\bigskip
		
\section{Appendix: The existence and uniqueness of strong solution}
\label{appendix}

Since the diffusion coefficient function $f(x) = -x \log(1 + |x|)$ in SDE \eqref{eq:critical sde} is only locally Lipschitz 
and satisfies the local growth condition in  \cite{MR2512800}, Theorem 6.2.11 in \cite{MR2512800} implies that 
\eqref{eq:critical sde} has a unique local solution.
The following theorem strengthens this result by proving the existence and uniqueness of strong global solution 
of SDE \eqref{eq:critical sde}.   	
\begin{theorem}
		\label{Thm 1 solution}
		SDE \eqref{eq:critical sde} has an unique strong global solution. Moreover, if $\alpha=2$, we have 
		$$\mathbb E |X_t|^2<\infty, \quad \forall \  t>0;$$
		if $\alpha \in (0,2)$, then for every $\beta \in (0,\alpha)$, 
		$$\mathbb E |X_t|^{\beta}<\infty, \quad \forall \  t>0.$$ 
	\end{theorem}
	
By adopting the argument in \cite[Section 1.6]{cerrai2001second}, we can show the theorem for the case of $\alpha=2$. 
This argument can be extended to prove the theorem for the case of $\alpha \in (0,2)$. However, we have not been 
able to find a proof in the literature. For completeness, we provide a proof of Theorem \ref{Thm 1 solution}. 
	
	
Since the coefficient function $f(x) = -x\log(1+\abs{x})$ is a local Lipschitz function,
For every $n \in \mathbb{N}_+$, we define the following truncated function $f_n(x)$ on $\bbR^d$ by
\[
		f_n(x) = \begin{cases}
			-x\log(1+\abs{x}),\ \  \ \ &\abs{x} \leqslant n ; \\
			\ -x\log(1+n),  &\abs{x}>n.
		\end{cases}
\]
	
\begin{proof}
By the definition of $f_n(x)$, it can been verified that $f_n(x)$ is a global Lipschitz  function with linear growth. Hence, 
\cite[Theorem 6.2.3]{MR2512800} implies that the SDE
		\begin{equation}
			\label{eq:trant. sde}
			\dd X_{n,t} = f_n(X_{n,t}) \dd t + \dd L_t
		\end{equation}
		has a unique strong solution $F_n$, and 
		\[
		X_n(t, x, \omega) := F_n(x, \omega) (t).
		\]
		Define a stopping time $\tau_n$ as
		\[
		\tau_n = \inf \left\{ t\geqslant 0: \abs{ X_{n,t} } \geqslant n   \right\}, \quad n\geqslant 2.
		\]
By the definition of $f_n(x)$, we have
		\[
		\begin{aligned}
		\inprod{x, f_n(x)} &= -\abs{x}^2 \log(1+\abs{x}) \mathbbm{1}_{[0,n]}(\abs{x}) - \log(1+n) \abs{x}^2 \mathbbm{1}_{(n,\infty)}(\abs{x}) \\
		&= -\abs{x}^2 \log(1+\abs{x}) \big( \mathbbm{1}_{[0,\rme-1]}(\abs{x}) +  \mathbbm{1}_{(\rme-1],n}(\abs{x}) \big) - \abs{x}^2 \mathbbm{1}_{(n,\infty)}(\abs{x}) \\ 
		&\leqslant -\abs{x}^2 + 1.
		\end{aligned}
		\]
		
In the case of $\alpha \in (0,2)$, we define function $V(x): \bbR^d \to \bbR$ as 
		\[
			V_{\beta}(x) = \left( 1+\abs{x}^2 \right)^{\beta/2}, \quad \beta \in (0,\alpha).
		\]
		Then we have 
		\[
			\begin{aligned}
				\nabla V_{\beta}(x) &= \frac{\beta x}{ ( 1+\abs{x}^2 )^{1-\beta/2 } }\ , \\
			\nabla^2 V_{\beta}(x) &= \frac{ \beta\ {I}_d }{ (1+\abs{x}^2)^{1-\beta/2} } + 
			\frac{ \beta (\beta-2) x x^{\prime} }{ (1+\abs{x}^2)^{2-\beta/2} }\ ,
			\end{aligned}
		\]
where $I_d$ is the identity matrix in $\bbR^{d\times d}$. Hence,  for all $x \in \bbR^d$, we have 
$\abs{x}^{\beta} \leqslant V_{\beta}(x) \leqslant 1+\abs{x}^{\beta}$ and 
		\[
			\abs{ \nabla V_{\beta}(x) } \leqslant \beta \abs{x}^{\beta-1}, 
			\quad
			\hsnorm{\nabla^2 V_{\beta}(x)} \leqslant \beta(3-\beta)\sqrt{d}.
		\]
	Besides, the following also holds for all $x\in\bbR^d$
		\[
			\inprod{\nabla V_{\beta}(x), f_n(x)} = \frac{\beta \inprod{x,f_n(x)}}{(1+\abs{x}^2)^{1-\beta/2}} \leqslant -\beta V_{\beta}(x) + 2\beta.
		\]
	 It\^o's formula yields that
			\begin{equation} \label{eq:ito of stable}
				\begin{aligned}
				V_{\beta}(X_{n,t}) &= V_{\beta}(x) + \int_0^t \inprod{ \nabla V_{\beta}(X_{n,s}) , f_n(X_{n,s}) } \dd s  \\
				+ & \int_0^t \int_{\abs{z}<1} \left[ V_{\beta}(X_{n,s}+z) - V_{\beta}(X_{n,s}) \right] \widetilde{P}(\dd s , \dd z) \\
				+ &  \int_0^t \int_{\abs{z}\geqslant1} \left[ V_{\beta}(X_{n,s}+z) - V_{\beta}(X_{n,s}) \right] P(\dd s , \dd z) \\
				  + &\int_0^t \int_{\abs{z}<1}\left[  V_{\beta}(X_{n,s}+z) - V_{\beta}(X_{n,s}) - \inprod{\nabla V_{\beta}(X_{n,s}), z} \right] \frac{C_{d,\alpha}\dd z \dd s}{\abs{z}^{d+\alpha}} .
			\end{aligned}
			\end{equation}
		
Before estimating $\bbE V_{\beta}(X_{n,t})$, we compute the following. If $\alpha \in (1,2)$, we let $\beta \in (1,\alpha)$, 
then we have that for any $x\in\bbR^d$
		\[
		\begin{aligned}
		&\pheq
		\int_{\bbR^d \backslash \{0\}} \left[ V_{\beta}(x+z) - V_{\beta}(x) - \inprod{ \nabla V_{\beta}(x), z } \mathbbm{1}_{(0,1)}(\abs{z}) \right] \frac{\dd z}{\abs{z}^{d+\alpha}} \\ 
		&= \int_{\abs{z}\geqslant1} \int_0^1 \inprod{ \nabla V_{\beta}(x+sz),z } \frac{\dd s \dd z}{\abs{z}^{d+\alpha}}  
				+  \int_{\abs{z}<1} \int_0^1 \int_0^s \hsprod{ \nabla^2 V_{\beta}(x+uz) , z z^{\prime} } \frac{\dd u \dd s \dd z}{\abs{z}^{d+\alpha}} \\
				&\leqslant  \beta \int_{\abs{z}\geqslant1} ( \abs{x}^{\beta-1}\abs{z} + \abs{z}^{\beta} ) \frac{ \dd z}{\abs{z}^{d+\alpha}} 
				+ \beta(3-\beta) \sqrt{d} \int_{\abs{z}<1} \abs{z}^2  \frac{ \dd z}{\abs{z}^{d+\alpha}} \\
				&= \beta s_{d-1} \left( \frac{\abs{x}^{\beta-1}}{\alpha-1} + \frac1{\alpha-\beta} \right) + \frac{\beta(3-\beta)s_{d-1}\sqrt{d}}{2(2-\alpha)} .
		\end{aligned}
		\]
On the other hand, for $\alpha \in(0,1]$, we use the inequality $(a+b)^{\beta} \leqslant a^\beta + b^{\beta}$, where $\beta \in (0,1)$ and $a, b>0$, to derive
		\[
			\begin{aligned}
				&\pheq
				\int_{\abs{z}\geqslant1} \left[ V_{\beta}(x+z) - V_{\beta}(x) \right] \frac{\dd z}{\abs{z}^{d+\alpha}} 
				\leqslant \int_{\abs{z}\geqslant1} \left[ 1+\abs{x+z}^\beta - \abs{x}^{\beta} \right] \frac{\dd z}{\abs{z}^{d+\alpha}} \\
				&\leqslant \int_{\abs{z}\geqslant1} \frac{1+\abs{z}^{\beta}}{\abs{z}^{d+\alpha}} \dd z 
				= \frac{(2\alpha-\beta) s_{d-1 } }{\alpha(\alpha-\beta)}.
			\end{aligned}
		\]
	Thus, an argument similar to that for the case $\alpha \in (1,2)$ implies that for $\alpha \in(0,1]$ and $\beta \in (0,\alpha)$,
		\[
		\begin{aligned}
		&\pheq
		\int_{\bbR^d \backslash \{0\}} \left[ V_{\beta}(x+z) - V_{\beta}(x) - \inprod{ \nabla V_{\beta}(x), z } \mathbbm{1}_{(0,1)}(\abs{z}) \right] \frac{\dd z}{\abs{z}^{d+\alpha}}\\
			&\leqslant \frac{(2\alpha-\beta) s_{d-1 } }{\alpha(\alpha-\beta)} + \frac{\beta(3-\beta)s_{d-1}\sqrt{d}}{2(2-\alpha)} .
		\end{aligned}
		\]
		
		By combining the above inequalities with equation \eqref{eq:ito of stable}, we  obtain that
			\begin{align}\label{ineq:d_ineq of stable}
				\frac{\dd \bbE V_{\beta}(X_{n,t})}{\dd t} 
				&= \bbE[ \inprod{ \nabla V_{\beta}(X_{n,t}), f_n(X_{n,t}) } ] \notag \\
				 + C_{d,\alpha} & \bbE \left[ \int_{\bbR^d \backslash \{0\}} [ V_{\beta}(X_{n,t} + z ) - V_{\beta}(z) - \inprod{ \nabla V_{\beta}(X_{n,t}), z } 
				 \mathbbm{1}_{(0,1)}(\abs{z}) ] \frac{\dd z}{\abs{z}^{d+\alpha}}  \right] \\
				&\leqslant -\beta \bbE V_\beta(X_{n,t}) + c_1\big[ \bbE V_{\beta}(X_{n,t})^{\frac{\beta-1}{\beta}} \mathbbm{1}_{(1,2)}(\alpha) + \mathbbm{1}_{(0,1]}(\alpha) \big] + c_2 \notag  \\ 
				&\leqslant -C_1 \bbE V_\beta(X_{n,t}) + C_2, \notag
			\end{align}
where the first inequality follows from $\abs{x}^{\beta-1} \leqslant V_{\beta}(x)^{(\beta-1)/\beta}$ as $\beta>1$,  the second inequality from H\"older's inequality, 
and $c_1,c_2,C_1,C_2$ are constants independent of $t$ and $n$. Hence, for any $\alpha \in (0,2)$, the differential inequality \eqref{ineq:d_ineq of stable} leads to
		\begin{equation}
			\label{ineq:bounds of V_beta}
			\bbE V_{\beta}(X_{n,t}) \leqslant V_\beta(x_0)\rme^{-C_1 t} + C_2/C_1 \leqslant C ,\quad \forall\ t\geqslant 0.
		\end{equation}
		Due to the definition of $V_{\beta}(x)$, we know that
		\[
			V_{\beta}(X_{n,\tau_n}) \geqslant (n^2 + 1)^{\frac{\beta}{2}},
		\]
		then for any $T>0$, \eqref{ineq:bounds of V_beta} and Markov's inequality lead to
		\[
			(n^2+1)^{\frac\beta2} \mathbb{P}(\tau_n \leqslant T) \leqslant \bbE\left[ V_{\beta}(X_{n,\tau_n}) \mathbbm{1}_{\tau_n \leqslant T} \right] \leqslant C.
		\]
		Let $n\to \infty$, we obtain 
		\[
			 \mathbb{P}(\lim_{n\to \infty} \tau_n \leqslant T) = 0,
		\]
	This and the arbitrariness of $T$ imply
		\begin{equation}
			\label{tau infty stable}
			\lim_{n\to \infty} \tau_n = \infty.
		\end{equation}
		As a result, \eqref{tau infty stable} yields that $\forall x\in \bbR^d,  t \geqslant 0$,
		\[
		X(t,x,\omega) = \lim_{n\to \infty} X_n(t,x,\omega)
		\]
		exists and is continuous with respect to $(t,x)$, which is a solution of SDE \eqref{eq:critical sde}.
		On the other hand, let $X(t)$ and $Y(t)$ be solutions with the same initial value $x$. Define
		\[
		\gamma_n = \inf\{ t: \abs{X(t)} \geqslant n \}, \quad \theta_n = \inf\{ t: \abs{Y(t)} \geqslant n \}.
		\]
		Then, we have
		\[
			X(\gamma_n \land \theta_n \land t) - Y(\gamma_n \land \theta_n \land t) = \int_0^{t} f_n(X(\gamma_n \land \theta_n \land s)) 
			- f_n(Y(\gamma_n \land \theta_n \land s)) \dd s,
		\]
		which implies that
		\[
			\begin{aligned}
				&\pheq
				\bbE \abs{ X(\gamma_n \land \theta_n \land t) - Y(\gamma_n \land \theta_n \land t) } \\
				& \leqslant \lipnorm{f_n} \int_0^t\bbE \abs{ X(\gamma_n \land \theta_n \land s) - Y(\gamma_n \land \theta_n \land s) } \dd s.
			\end{aligned}
		\]
		Then by Gr\"onwall's inequality, we have that
		\[
			\bbE \abs{ X(\gamma_n \land \theta_n \land t) - Y(\gamma_n \land \theta_n \land t)  } = 0, \quad \forall \ t\geqslant 0,
		\]
		which implies that $X(t) = Y(t)$ on $t\leqslant \gamma_n \land \theta_n$. Then, let $n\to \infty$, we have $\gamma_n \land \theta_n 
		\to \infty , a.s.$. Hence, $X(t) = Y(t) , a.s.$ for all $t \geqslant 0$. 
		
		Finally, for the moment estimation, by the same argument for bounding $\bbE V_{\beta}(X_{n,t})$, we can establish that 
		$\bbE V_{\beta}(X_t) \leqslant C$ for all $t>0$, where $C>0$ is a constant not depending on $t$. This implies 
		that $\bbE \abs{X_t}^{\beta} \leqslant \bbE V_{\beta}(X_t) \leqslant C$ for all $\beta \in (0,\alpha)$ and $t>0$. 
	        This completes the proof.
	\end{proof}

	\section*{Acknowledgements}
	Y. Xiao is supported in part by the NSF grant DMS-2153846. L. Xu is supported by the National Natural Science Foundation of 
	China No. 12071499, The Science and Technology Development Fund (FDCT) of Macau S.A.R. FDCT 0074/2023/RIA2, and the
	University of Macau grants MYRG2020-00039-FST, MYRG-GRG2023-00088-FST.

\bibliographystyle{plain}

\end{CJK}
\end{document}